\documentclass[12pt]{article}
\usepackage{amsmath,amsthm,amssymb,latexsym,color}
\usepackage{bbm}
\usepackage{hyperref}
\theoremstyle{plain}
\newtheorem{theor}{Theorem}[section]
\newtheorem{observation}{Observation}
\newtheorem*{example}{Example}
\newtheorem{prop}[theor]{Proposition}
\newtheorem{cor}[theor]{Corollary}
\newtheorem{lemma}[theor]{Lemma}
\theoremstyle{remark}
\newtheorem{rem}[theor]{Remark}

\def\R{{\mathbb R}}
\def\N{{\mathbb N}}
\def\Z{{\mathbb Z}}
\def\Exp{{\mathbb E}}

\def\spn{{\rm span}}
\def\dist{{\rm dist}}
\def\Event{{\mathcal E}}
\def\Proj{{\rm Proj}}
\def\Prob{{\mathbb P}}
\def\tuples{{\mathcal T}}
\def\col{{\rm Col}}
\def\Id{{\rm Id}}
\def\Net{{\mathcal N}}
\def\Comp{{\rm Comp}}
\def\Class{\mathcal C}
\def\argmax{{\rm argmax}}
\def\rhonew{{\mathcal P}}
\def\Graph{{\mathcal G}}
\def\offset{{\rm ofst}}
\def\vl{{\rm val}}
\def\mindist{{\rm mndst}}
\def\kmax{{\text{-}\hspace{-0.1cm}\max}}

\title{Invertibility via distance for non-centered random matrices with continuous distributions}
\author{Konstantin Tikhomirov\footnote{Princeton University, NJ; email: kt12@math.princeton.edu.}}

\begin{document}

\maketitle
\begin{abstract}
Let $A$ be an $n\times n$ random matrix with independent rows $R_1(A),\dots,R_n(A)$, and assume that
for any $i\leq n$ and any three-dimensional linear subspace $F\subset\R^n$ the orthogonal
projection of $R_i(A)$ onto $F$ has distribution density $\rho(x):F\to\R_+$ satisfying $\rho(x)\leq C_1/\max(1,\|x\|_2^{2000})$ ($x\in F$)
for some constant $C_1>0$.
We show that for any fixed $n\times n$ real matrix $M$ we have
$$\Prob\{s_{\min}(A+M)\leq t n^{-1/2}\}\leq C'\, t,\quad\quad t>0,$$
where $C'>0$ is a universal constant.
In particular, the above result holds if the rows of $A$ are
independent centered log-concave random vectors with identity covariance matrices.
Our method is free from any use of covering arguments,
and is principally different from a standard approach involving a decomposition of the unit sphere
and coverings, as well as an approach of Sankar--Spielman--Teng for non-centered Gaussian matrices.
\end{abstract}



\tableofcontents

\section{Introduction}

We recall that, given an $n\times n$ matrix $A$ (we will always assume the entries are real-valued),
its largest and smallest singular value
can be defined as
$$s_{\max}(A):=\sup_{x\in S^{n-1}}\|Ax\|_2,\quad\quad s_{\min}(A):=\inf_{x\in S^{n-1}}\|Ax\|_2,$$
and the condition number of $A$ is given by $\kappa(A):=s_{\max}(A)/s_{\min}(A)$.

It is known that the condition number of $A$ can be used as a measure of loss of precision
in numerical algorithms solving linear systems with the coefficient matrix $A$ \cite{Smale}.
A natural way to study the set of all well-conditioned $n\times n$ matrices is to consider a
probability measure on $\R^{n\times n}$ and analyze the distribution of the condition number with respect to that measure
(so that the input coefficient matrix is random).
For the standard $n\times n$ Gaussian matrix $G$ (with i.i.d.\ $N(0,1)$ entries), Edelman \cite{Edelman} 
computed limiting distributions of the appropriately rescaled smallest singular value $s_{\min}(G)$
and the condition number $\kappa(G)$,
and, in particular, verified that $\kappa(G)\lesssim n$ with high probability confirming an old conjecture
due to von Neumann and Goldstine \cite{Neumann}.
As a culmination of work of many researchers (see \cite{K,KKS,LPRT,TV,R} and references therein)
Rudelson and Vershynin \cite{RV_square} showed for any random matrix $A$ with i.i.d.\ {\it subgaussian} real entries that
$\Prob\{s_{\min}(A)\leq t\,n^{-1/2}\}\leq Ct+c^n$, implying
$\Prob\{\kappa(A)\geq t n\}\leq C/t+c^n$ ($t>0$). Here $C>0$ and $c\in(0,1)$
depend only on the subgaussian moment.
The universality of the limiting distribution of the condition number was confirmed by Tao and Vu \cite{TV-limiting}
(see \cite{BP} for an earlier result in this direction).
We refer to \cite{AGLPT,RT} for some extensions of those results allowing heavy tails and some dependencies
within the columns, as well as to \cite{BR-smin,BR-circ} for the condition number of sparse matrices with independent entries.

Let us note here that estimates for the condition number of matrices of the form $A-z\Id$, where $z$ is a fixed complex number,
play a crucial role in establishing limiting laws for the empirical spectral distribution,
we refer, in particular, to \cite{BC12,RV_JAMS,BR-circ,C-smin} and references therein.

The study of the condition number of a shifted random matrix was put forward by Spielman and Teng \cite{ST}
as a way to find a balance between the worst-case and the average-case analysis of algorithms.
In the context of invertibility, the goal would be
to show that a neighborhood of every fixed matrix consists largely of well-conditioned matrices.
In \cite{SST}, Sankar, Spielman and Teng proved that for the standard $n\times n$ real Gaussian matrix $G$
and any fixed matrix $M$ one has $\Prob\{s_{\min}(G+M)\leq t n^{-1/2}\}\leq Ct$ for all $t>0$
and some universal constant $C>0$.
In particular, for any $M\in\R^{n\times n}$ with probability, say, $99$ per cent the condition number of $G+M$ is bounded from above by
a constant multiple of $\sqrt{n}s_{\max}(M)+n$. Interestingly, this estimate
does not extend to discrete random matrices;
in particular, for any $L\geq n$ one can find a fixed matrix $M$ with $s_{\max}(M)=L$
such that, for the standard Bernoulli ($\pm 1$) random $n\times n$ matrix $B$ one has
$\Prob\{s_{\min}(B+M)\leq Cn/L\}\geq 1/5$ and $\Prob\{\kappa(B+M)\geq c' L^2/n\}\geq 0.19$
(see paper \cite{TV-cn} of Tao and Vu, as well as Observation~\ref{p:shifted Bernoulli} of this paper).
This shows that estimates for the condition number in the Bernoulli model must be strictly weaker than
those available in the Gaussian setting. In \cite{TV-cn}, Tao and Vu showed that
$\Prob\{s_{\min}(B+M)\leq n^{-C(2t+2)+1/2}\}\leq n^{-t+o(1)}$ whenever
$t\leq C$ and $s_{\max}(M)\leq n^C$ for some (arbitrarily large) constant $C\geq 1/2$,
in particular, the condition number of the shifted Bernoulli matrix $B+M$ is bounded from above by
$n^{3C-1/2}\,n^{o(1)}$ with probability $0.99$.
The same estimate is true for $B$ replaced with any matrix $A$ with appropriately scaled i.i.d.\ entries
satisfying certain moment assumptions \cite[Corollary~3.5]{TV-cn}.
On the other hand, when the entries of the matrix $A$ are jointly independent and have uniformly bounded densities,
it is a simple observation that $\Prob\{s_{\min}(A+M)\leq t n^{-1}\}\leq vt$ ($t>0$), where $v>0$ may depend only on the
upper bound for the densities and not on $M$ (see Observation~\ref{o: bounded density trivial}).
Compared to the Sankar--Speilman--Teng bound for the Gaussian matrices, this last estimate is worse by the factor $\sqrt{n}$.
Let us note at this point that, in the i.i.d.\ model with appropriate tail decay conditions on the entries,
estimating the condition number essentially reduces to estimating the smallest singular value of the
matrix since bounds for the largest singular value (the spectral norm) are well known
(see, in particular, \cite{YBK,L}).

\bigskip

The above results left open the question whether the discontinuous distribution
of the entries in the shifted Bernoulli model is the only obstruction to bounding
the condition number by the same quantity as in the shifted Gaussian model.
In other words, the question is whether one can define a class of ``sufficiently smooth'' distributions
(with appropriate tail decay conditions) which enjoy the same
estimates for the condition number as the shifted Gaussian matrix.
We remark here that the argument of Sankar--Spielman--Teng \cite[Theorem~3.3]{SST}
relies on orthogonal invariance of the Gaussian distribution,
and its applicability to models
with no rotation invariance is unclear.
In the same paper \cite{SST}, the authors assumed that the shift-independent small ball probability estimates
for $s_{\min}$
matching those observed for the shifted Gaussian matrices, should
hold for some other classes of distributions \cite[Section~7.1]{SST}.
However, to our best knowledge, no analysis of this problem has been performed in literature.
Further, let us note that in case of shifted {\it symmetric} random matrices some recent results are available 
in the Gaussian setting \cite{Bourgain} (see also a related work \cite{A et al})
and for arbitrary continuous distributions \cite{FV}, however, the estimates obtained in paper \cite{FV}
are worse than in the Gaussian case.

A closely related (but simpler) problem which we consider in the first part
of this paper is the additive term
in the small ball probability estimate of the smallest singular value of {\it centered} random matrices in the Rudelson--Vershynin
work \cite{RV_square} and related papers (including \cite{AGLPT,RT}).
Recall that \cite{RV_square} gives an estimate of the form $\Prob\{s_{\min}(A)\leq t\,n^{-1/2}\}\leq Ct+c^n$ ($t>0$)
assuming appropriate scaling of the entries of $A$.
It is easy to see that for the Bernoulli model when the singularity probability is exponential in $-n$,
the additive term $c^n$ in the small ball probability estimate cannot be removed.
However, when the distributions are ``sufficiently smooth'',
one may expect that the term becomes redundant as in the Gaussian case.

\bigskip

Let us turn to a description of our results.
In this paper, we will impose several conditions on the distribution of a matrix.
Fix an integer $m\geq 1$ and a positive real number $K$. Then we write that a random vector $X$ in $\R^n$ 
satisfies condition \eqref{eq: density r-condition} with parameters $m$ and $K$ if 
\begin{equation}\label{eq: density r-condition}\tag{\bf{C1}}
\begin{split}
&\mbox{For any $m$-dimensional subspace $F\subset\R^n$, the random vector $\Proj_F(X)$}\\
&\mbox{has distribution density (viewed as a function on $F$) bounded above by $K$.}
\end{split}
\end{equation}
(Here and further in the text we denote by $\Proj_F$ the orthogonal projection onto $F$.)

In the first (simpler) part of the paper, we revisit the case when the shift is zero.
Recall that a random vector $X$ in $\R^n$ is called {\it isotropic} if $\Exp X=0$ and the covariance matrix ${\rm Cov}X=\Id$.
We prove
\begin{theor}\label{th:revisiting}
There are universal constants $n_0\in\N$ and $C>0$ with the following property.
Let $n\geq n_0$ and let $A$ be an $n\times n$ random matrix with independent
isotropic rows so that each row satisfies
\eqref{eq: density r-condition} for $m=1$, $m=4$ and some $K>0$.
Then for every $t>0$ we have
$$\Prob\big\{\|A^{-1}\|_{HS}\geq tn^{1/2}\big\}\leq \frac{C K}{t}.$$
\end{theor}
Here, $\|\cdot\|_{HS}$ is the Hilbert--Schmidt norm.
Note that $\|A^{-1}\|_{HS}\geq s_{\max}(A^{-1})=1/s_{\min}(A)$ deterministically,
so the above relation is directly translated into the small ball probability estimate for $s_{\min}$.

As an illustration of the theorem, let us assume that a random $n\times n$ matrix $A$ (for $n$ large) has jointly independent entries,
each of mean zero and variance one, and, moreover, the distribution density of each entry
is uniformly bounded by a number $L>0$. Then a result of \cite{RV_images}
implies that the rows of $A$ satisfy condition \eqref{eq: density r-condition} for any $m\in\N$ and
$K:=(CL)^m$ for a universal constant $C>0$. In particular, $A$ satisfies the assumptions of Theorem~\ref{th:revisiting},
and we obtain
\begin{cor}
For any $L>0$ there is $v=v(L)>0$ depending only on $L$ with the following property.
Let $n$ be a positive integer and let $A$ be an $n\times n$ random matrix with independent entries,
each of zero mean, unit variance, and with the distribution density bounded above by $L$. Then
$$\Prob\big\{s_{\min}(A)\leq tn^{-1/2}\big\}\leq vt,\quad t>0.$$
\end{cor}

This recovers the Rudelson--Vershynin theorem \cite{RV_square} (and its heavy-tailed extension \cite{RT})
under the additional assumption of the bounded density but does not involve the
additive term $c^n$. As we mentioned before,
the standard (by now) approach to invertibility of random matrices involves some kind of decomposition
of the unit sphere with subsequent application of covering arguments.
In contrast, the proof
of Theorem~\ref{th:revisiting} is based on studying correlations between $\dist\big(R_{i}(A),\spn\{R_j(A),\,j\neq i\}\big)$,
related to $\|A^{-1}\|_{HS}$ via the standard relation (negative second moment identity)
$$
\|A^{-1}\|_{HS}^2=\sum\limits_{i=1}^n \|\col_{i}(A^{-1})\|_2^2=\sum\limits_{i=1}^n
\dist\big(R_{i}(A),\spn\{R_j(A),\,j\neq i\}\big)^{-2}.
$$
The idea behind the proof of Theorem~\ref{th:revisiting}
was the first step in our study of the shifted matrices.
Let us note that a crucial aspect of the Rudelson--Vershynin theorem --- analysis of the arithmetic structure
of normal vectors to hyperplanes spanned by the matrix rows --- does not play a role in this paper
in view of the bounded density assumption.

As another illustration of the theorem, let us consider matrices with independent log-concave rows.
Recall that a random vector $X$ in $\R^n$ is {\it log-concave} if for any two compact sets $S_1,S_2$
in $\R^n$ and any $\lambda\in[0,1]$ we have
$$\Prob\{X\in(1-\lambda)S_1+\lambda S_2\}\geq \Prob\{X\in S_1\}^{1-\lambda}\,\Prob\{X\in S_2\}^\lambda.$$
It can be shown that any isotropic log-concave vector in $\R^n$ satisfies condition \eqref{eq: density r-condition}
for any $m\geq 1$ and $K=K(m)>0$ depending only on $m$
(see, for example, \cite[Chapter~2]{BGVV} for details).
Thus, for an $n\times n$ matrix $A$ with independent centered isotropic rows we have
$\Prob\big\{s_{\min}(A)\leq t\,n^{-1/2}\big\}\leq Ct$ for all $t>0$.
This result was verified in \cite{AGLPT}
with the additional restriction $t\geq e^{-c\sqrt{n}}$. Moreover, the proof of \cite{AGLPT}
relies on two highly non-trivial results \cite{Paouris} and \cite{ALPT} concerning properties of log-concave distributions
and approximation of their covariance matrices.
On the contrary, our much less involved technique provides the deviation estimate
for all $t>0$ thereby answering a question from \cite{AGLPT}.

\bigskip

Let us turn to the main object of the paper --- shifted random matrices.
We say that an $n$-dimensional vector $X$ satisfies
condition \eqref{C2} with parameters $K_1,K_2>0$ if
\begin{equation}\tag{\bf C2}\label{C2}
\begin{split}
&\mbox{for any $3$-dimensional
subspace $F\subset\R^n$, the distribution density $\rho(x)$}\\
&\mbox{of the vector $\Proj_F(X)$ satisfies}
\;\;\rho(x)\leq K_1 /\max(1,\|x\|_2^{K_2}),\quad x\in F.
\end{split}
\end{equation}

\medskip

The main result of this paper is the following theorem:
\begin{theor}\label{t:shift}
For any $K_1>0$ there are $n_0=n_0(K_1)\in\N$ and $v=v(K_1)>0$ with the following property.
Let $n\geq n_0$ and let $A$ be an $n\times n$ random matrix with independent rows, such that each
row satisfies condition \eqref{C2} with parameters $K_1$ and $K_2:=2000$. Then for any fixed $n\times n$ matrix $M$ we have
$$\Prob\big\{\|(A+M)^{-1}\|_{HS}\geq t\sqrt{n}\big\}\leq \frac{v}{t},\quad t>0.$$
\end{theor}
Clearly, $K_2=2000$ can be replaced with any number greater than $2000$. We did not try to optimize this aspect
of the theorem (see remark at the end of the paper).
Note that the small ball probability estimates implied by the theorem match the Gaussian case
\cite{SST}. It can be checked that condition \eqref{C2}
is satisfied by any log-concave isotropic random vector for arbitrarily large $K_2$ and for some $K_1=K_1(K_2)$
(we refer to \cite{BGVV} for details).
Thus, we have
\begin{cor}
There are universal constants $n_0\in\N$ and $C>0$ with the following property.
Let $n\geq n_0$, let $A$ be an $n\times n$ random matrix with independent
isotropic log-concave rows,
and let $M$ be any fixed $n\times n$ matrix. Then
$$\Prob\big\{s_{\min}(A+M)\leq t\, n^{-1/2}\big\}\leq C\,t,\quad t>0.$$
\end{cor}

\bigskip

Our proof of Theorem~\ref{t:shift} involves estimating correlation between columns of the inverse matrix, as well as
some combinatorial arguments. The main conceptual part of the proof consists in defining
{\it an $(\alpha,\eta)$--structure} --- a collection of special events, together with several random variables
on a product probability space --- which provides a way to (implicitly) estimate the probability of the ``bad'' event
that the Hilbert--Schmidt norm of the inverse is large.
At the beginning of Sections~\ref{s:revisited} and~\ref{s:shifted}, we will discuss in more detail the
limitations of the standard approach based on covering arguments,
giving a justification for our alternative approach.

\section{Notation}\label{s:prelim}

Given a real number $r$, we define $\lfloor r\rfloor$ as the largest integer less than or equal to $r$
and $\lceil r\rceil$ as the smallest integer greater or equal to $r$.
For a positive integer $m$, $[m]$ stands for the set $\{1,2,\dots,m\}$.
The cardinality of a finite set $I$ is denoted by $|I|$.
For a collection of vectors $x_1,\dots,x_m$ in $\R^n$, denote by $\spn\{x_1,\dots,x_m\}$
the linear span of $x_1,\dots,x_m$ ($n$ shall always be clear from the context).
Let $\langle\cdot,\cdot\rangle$ be the standard inner product in $\R^n$. 
For a vector $y$ in $\R^n$ and a linear subspace $F\subset\R^n$, let $\Proj_F(y)$
stand for the orthogonal projection of $y$ onto $F$. By $F^\perp$ we denote
the orthogonal complement of $F$ in $\R^n$.
The standard vector basis in $\R^n$ is $e_1,e_2,\dots,e_n$.

For any $p\geq 1$, let $\|\cdot\|_p$ denote the standard $\ell_p^n$--norm in $\R^n$.
The Hilbert--Schmidt norm of a matrix $B=(b_{ij})$ is defined as
$$\|B\|_{HS}:=\sqrt{\sum_{i,j}{b_{ij}}^2}$$
(it is easy to see that always $\|B\|_{HS}\geq s_{\max}(B)$).

The rows of a matrix $B$ are denoted by $R_1(B),\dots$, $R_n(B)$ and columns by $\col_1(B)$, $\dots$, $\col_n(B)$.
Given an $n\times n$ matrix $B$ (whether deterministic or random), denote by
$H^i(B)$ ($i=1,2,\dots,n$) the linear subspace spanned by vectors $\{R_j(B)\}_{j\neq i}$.
More generally, if $I$ is any non-empty subset of $[n]$, we denote by $H^I(B)$
the subspace spanned by vectors $\{R_j(B):\,j\in[n]\setminus I\}$.
The Euclidean distance from a point $x\in\R^n$ to a set $S\subset\R^n$ will be denoted by
$\dist(x,S)$.

For every integer $r\geq 1$ let $\tuples_{n,r}$ be the collection of all subsets of $[n]$ of cardinality $r$.

Universal constants are denoted by $C,c$, etc., and their value may be different from line to line.
For two strictly positive quantities $a,b$, we write $a\approx b$ if $C^{-1}a\leq b\leq Ca$
for a universal constant $C>0$. Further, we write $a\gtrsim b$ ($a\lesssim b$)
if $Ca\geq b$ (resp., $a\leq Cb$) for some constant $C>0$.

We let $(\Omega,\Prob)$ be the underlying probability space
(for a random matrix having independent rows, we will sometimes interpret $(\Omega,\Prob)$
as the product space $\prod_{i=1}^n(\Omega_i,\Prob_i)$, with $\Omega_i$ being the domain of the $i$-th row).
Given a random variable (random matrix, random set, etc.) $\xi$ on $\Omega$ and a point $\omega\in\Omega$,
by $\xi(\omega)$ we denote its realization
at $\omega$.

\section{Invertibility of centered random matrices, revisited}\label{s:revisited}

The quantitative approach to invertibility developed to a large extent in \cite{LPRT,TV,R,RV_square}
and later reapplied in many works within the random matrix theory,
involves some decomposition of $\R^n$ into sets of ``more structured'' and ``less structured''
vectors, and estimating the infimum of the matrix-vector product over the structured vectors
via a covering argument. In particular, in paper \cite{RV_square} dealing with
an $n\times n$ random matrix $A$ with i.i.d.\ entries, the sphere $S^{n-1}$ is split into {\it compressible}
(close to sparse)
and {\it incompressible} (far from sparse) vectors; the infimum of $\|Ax\|_2$ over compressible
unit vectors $x\in{\Comp}$ is then bounded by $\inf_{x\in\Net}\|Ax\|_2-\varepsilon\, s_{\max}(A)$,
where $\Net$ is an $\varepsilon$-net on ${\Comp}$ (for an appropriately chosen $\varepsilon$).
The relatively small size of the set 
${\rm Comp}$ allows to beat the cardinality of the net $\Net$ by small ball probability estimates for $\|Ax\|_2$
for individual vectors $x\in\Comp$, which, together with an upper bound for $s_{\max}(A)$,
yields a satisfactory lower bound for $\inf_{x\in\Comp}\|Ax\|_2$.
For heavy-tailed matrices, when a strong estimate of $s_{\max}(A)$
is not available, a more elaborate version of the covering argument was developed in \cite{RT};
still, the basic idea of balancing the size of a covering with probability estimates for individual vectors
remained in place.

One of the consequences of this approach is presence of an additive term in the small ball probability estimate
for $s_{\min}(A)$
which essentially encapsulates the event that a satisfactory bound for $\inf_{x\in\Comp}\|Ax\|_2$
{\it via the covering argument} fails.
In the i.i.d.\ model (with zero mean entries of unit variance) the term is exponentially small in $n$ \cite{RV_square}
(see also \cite{RT}).
In work \cite{AGLPT} dealing with matrices with independent isotropic log-concave columns, the authors
got an estimate $\Prob\{s_{\min}(A)\leq t\,n^{-1/2}\}\leq Ct+\exp(-c\sqrt{n})$ ($t>0$),
where the additive term $\exp(-c\sqrt{n})$ again comes from the ``imperfections'' in the covering argument
for compressible vectors.
Of course, for some classes of distributions (including, for example, the Bernoulli variables)
an additive term in the small ball estimate {\it must} be present. For example,
in the Bernoulli case, the small ball probability estimate given by \cite{RV_square} is
$\Prob\{s_{\min}(A)\leq t\,n^{-1/2}\}\leq Ct+c^n$,
where the additive term $c^n$ must be greater than $2^{-n}$, accounting, in particular,
for matrix realizations with two columns equal.
At the same time, when the distribution of the entries is sufficiently smooth, it is natural to expect
that the term should disappear.

In this section, we will prove that for distributions with bounded densities (satisfying some
mild moment assumptions) this additive term indeed can be removed, yielding
a small ball probability estimate of the form $\Prob\{s_{\min}(A)\leq t\,n^{-1/2}\}\leq Ct$, $t>0$.
As we already mentioned in the introduction, the primary value of this result lies in
the method for bounding the smallest singular value which is completely free from any covering arguments.
The method can be viewed as a development of the approach to bounding the infimum of $\|Ax\|_2$
over {\it incompressible} vectors in paper \cite{RV_square}, which was based only on studying the
distances between the matrix rows (see \cite[Lemma~3.5]{RV_square}).
Let us note that the crucial (and highly non-trivial) part of paper \cite{RV_square}
consists in proving good anti-concentration estimates for the distances, which involved
studying the arithmetic structure of the normal vector to a subspace spanned by $n-1$ rows.
The paper \cite{RV_square} of Rudelson and Vershynin thus presents a major contribution to what is now called
the Littlewood--Offord theory, following earlier
works of Tao and Vu (see \cite{TV}).
For us, the assumption of bounded density of distributions of the rows' projections
allows us to completely avoid any Littlewood--Offord--type arguments in our proof,
and, in this respect, Theorem~\ref{th:revisiting} is fundamentally simpler.

As we already mentioned in the introduction, the main part of the proof of Theorem~\ref{th:revisiting}
essentially consists in showing that the Euclidean norms of columns of the inverse matrix are highly correlated.
The strong correlation would imply that the situation when the norm of some of the columns is relatively large
while norms of the others are relatively small, is very unlikely: the typical case corresponds to all columns
having more or less the same Euclidean length. This, combined with a (technically) simple
averaging argument and Markov's inequality, completes the work.
In may be useful to consider the following elementary example.
\begin{example}[Column correlations] Let ${G_2}$ be a $2\times 2$ standard real Gaussian matrix,
and let $\col_1,\col_2$ be the columns of ${G_2}^{-1}$.
Since the rows of $G_2$ form a biorthogonal system with $\{\col_1,\col_2\}$, we have
$\|\col_i\|_2=\dist(R_i(G_2),\spn\{R_{3-i}(G_2)\})^{-1}$, $i=1,2$.
Thus, $\|\col_i\|_2^{-1}$ ($i=1,2$) are distributed as the absolute value of a standard Gaussian
variable, and, in particular, $\Prob\{\|\col_2\|_2\leq 1\}\approx 1$
and $\Prob\{\|\col_1\|_2\geq q\}\gtrsim q^{-1}$ (for any large enough parameter $q$).
Now, let us estimate the intersection $\Event_{\cap}$ of the two events: the probability that $\|\col_2\|_2\leq 1$
and $\|\col_1\|_2\geq q$. Inside this event, we have
$$q\,\dist(R_2(G_2),\spn\{R_{1}(G_2)\})\leq \dist(R_1(G_2),\spn\{R_{2}(G_2)\}),$$
whence necessarily $\|R_1(G_2)\|_2\geq q\|R_2(G_2)\|_2$. A rough estimate yields
\begin{align*}
\Prob\big\{&\|R_1(G_2)\|_2\geq q\|R_2(G_2)\|_2\big\}\\
&\leq \Prob\big\{\|R_1(G_2)\|_2\geq 2\sqrt{2}\sqrt{\log q}\big\}
+\Prob\big\{\|R_2(G_2)\|_2\leq 2\sqrt{2}\sqrt{\log q}/q\big\}\\
&\lesssim \frac{\log q}{q^2}.
\end{align*}
Thus, the event $\Event_\cap$ has a much smaller probability than the product
$\Prob\{\|\col_2\|_2\leq 1\}\,\cdot\,\Prob\{\|\col_1\|_2\geq q\}$,
which can be viewed as a consequence of high colleration between the lengths of $\col_1$ and $\col_2$.
\end{example}

\subsection{Auxiliary deterministic lemmas}

In this subsection, we verify some simple linear algebraic and combinatorial properties
of vector sets that will be employed later in probabilistic context.
Let us refer to \cite[Lemma~C.1]{TV-limiting} which is close to lemmas proved here.
The authors of \cite{TV-limiting} used their Lemma~C.1 with the same purpose ---
to show that the columns of the inverse matrix are correlated \cite[Proposition~3.2]{TV-limiting},
although, being an auxiliary result, those estimates
are not strong enough to imply our Theorem~\ref{th:revisiting}.

\begin{lemma}\label{l: aux9876}
Let $r\geq 2$ and let $x_1,x_2,\dots,x_r$ be vectors in $\R^r$. Further, assume that for some $a,b> 0$
we have
$$\dist(x_1,\spn\{x_2,\dots,x_r\})\leq a\quad\mbox{ and }\quad\dist (x_i,\spn\{x_j:\,j\neq i\})\geq b,\;\;2\leq i\leq r.$$
Then there is $i_0\in\{2,3,\dots,r\}$ such that $\|x_{i_0}\|_2\geq \frac{b}{2ar}\|x_1\|_2$.
\end{lemma}
\begin{proof}
Note that $\|x_i\|_2\geq b$ for all $i\geq 2$, so the statement is obvious whenever $\|x_1\|_2\leq 2ar$.
Now, assume that $\|x_1\|_2\geq 2ar$.
By the conditions of the lemma, there are real numbers $\alpha_2,\alpha_3,\dots,\alpha_r$
such that
\begin{equation}\label{eq: aux 23-981}
\Big\|x_1-\sum_{i=2}^r \alpha_i x_i\Big\|_2\leq a.
\end{equation}
Let $i^*$ be the index corresponding to the largest $\alpha_i$ (by the absolute value).
From the last inequality we have
$$\dist(\alpha_{i^*}\,x_{i^*},\spn\{x_j:\,j\neq i^*\})\leq a,$$
whence, in view of the conditions of the lemma, $|\alpha_{i^*}|\leq a/b$.
Now, since $\|x_1\|_2\geq 2a$ and in view of \eqref{eq: aux 23-981}, we have
$$\frac{ar}{b}\max_{i\geq 2}\|x_i\|_2
\geq r|\alpha_{i^*}|\max_{i\geq 2}\|x_i\|_2\geq\Big\|\sum_{i=2}^r \alpha_i x_i\Big\|_2\geq \frac{1}{2}\|x_1\|_2.$$
The result follows.
\end{proof}
\begin{rem}
Note that for $r=2$ the above lemma simply states that, assuming that
$\dist(x_1,\spn\{x_2\})\leq a$ and $\dist (x_2,\spn\{x_1\})\geq b$,
we necessarily have $\|x_2\|_2\gtrsim \frac{b}{a}\|x_1\|_2$, which, if we treat $x_1,x_2$
as rows of a $2\times 2$ matrix, corresponds to the situation considered in the basic example
at the beginning of the section.
\end{rem}

\begin{lemma}\label{l: Q1Q2}
Let $n\geq r\geq 1$, $a,b>0$, let $B$ be an $n\times n$ deterministic matrix,
and $I,J$ be disjoints subsets of $[n]$
such that $|J|\geq \frac{n}{2}$ and
$$\dist(R_i(B),H^i(B))\leq a,\;\;i\in I;\quad\quad \dist(R_i(B),H^i(B))\geq b,\;\;i\in J.$$
Fix any $\tau>0$ and denote
\begin{align*}
Q_1(B)&:=\big\{S\in \tuples_{n,r}:\,\exists j\in S\cap I\mbox{ such that }\dist(R_j(B),H^S(B))\leq \tau\big\};\\
Q_2(B)&:=\Big\{S\in \tuples_{n,r}:\,S\cap I\neq \emptyset\mbox{ and }
\exists j\in S\setminus I\mbox{ such that }\dist(R_j(B),H^S(B))\geq \frac{\tau b}{2ar}\Big\};
\end{align*}
Then necessarily
$$|Q_1(B)\,\cup\,Q_2(B)|\geq |I|{\lceil n/2\rceil\choose r-1}.$$
\end{lemma}
Before proving the lemma, let us describe the idea behind it.
Ultimately, we want to show that the
columns of the inverse of our random matrix
are highly correlated in a sense we discussed above.
By bi-orthogonality of the rows of the
original matrix and the columns of the inverse matrix, this amounts to checking that the distances from a row of the original matrix
to the span of the remaining rows typically have comparable order of magnitude.
The set of indices $J$ in the above lemma will correspond
to rows having {\it relatively big} distances to respective linear spans, and the set $I$ to rows with
small distances to the spans.
Thus, we will need to show that the situation when both $J$ and $I$ are {\it large} is unlikely (has small probability).
Lemma~\ref{l: Q1Q2}
states that, as long as cardinality of $I$ is large, so is the sum of cardinalities of the auxiliary
sets of $r$-tuples $Q_1(B)$ and $Q_2(B)$. The set $Q_1(B)$, roughly speaking, accounts for the situation
when the projection of a row onto the orthogonal complement of $n-r$ rows has {\it small} Euclidean norm.
Since here we are dealing with the projection onto an $r$-dimensional subspace, the anti-concentration
is much stronger than in one-dimensional setting, and we can show that typically the cardinality of $Q_1(B)$
is small (see probabilistic lemmas below).
The set $Q_2(B)$ accounts for the situation when projection of a row onto the orthogonal complement of $n-r$
rows is {\it large}. This time, we will use the moment assumptions to show that this situation is also unlikely,
whence $Q_2(B)$ typically has small cardinality.
Then the lower bound for the sum of cardinalities in the above lemma will indicate that the situation when both
$J$ and $I$ are large is atypical.
Observe that, in the case $r=2$, the treatment of the auxiliary sets $Q_1(B)$ and $Q_2(B)$ essentially corresponds to
the way we estimated $\Prob\big\{\|R_2(G_2)\|_2\geq q\|R_1(G_2)\|_2\big\}$ in the example above.

\smallskip

\begin{proof}[Proof of Lemma~\ref{l: Q1Q2}]
Obviously, the number of subsets $S$ of $[n]$ of cardinality $r$ having exactly one element from $I$ and the other $r-1$
elements from $J$, is at least $|I|{\lceil n/2\rceil\choose r-1}$. It remains to show that any such subset $S$
belongs to the union $Q_1\cup Q_2$. Indeed, assume that $S\notin Q_1$.
Then the unique element $j\in S\cap I$ satisfies $\dist(R_j(B),H^S(B))> \tau$. Next, denoting the orthogonal projection
of $R_i(B)$ onto $(H^S(B))^\perp$ by $x_i$ ($i\in S$), we have
$$\dist(x_i,\spn\{x_k:\,k\in S\setminus\{i\}\})=\dist(R_i(B),H^i(B))\quad\begin{cases}\leq a,&\mbox{if }i=j\\\geq b,&\mbox{if }i\neq j.
\end{cases}$$
Thus, in view of Lemma~\ref{l: aux9876}, there is $i_0\in S\setminus\{j\}$ such that
$$\dist(R_{i_0}(B),H^S(B))=\|x_{i_0}\|_2\geq \frac{b}{2ar}\|x_j\|_2=\frac{b}{2ar}\dist(R_j(B),H^S(B))\geq \frac{\tau b}{2ar}.$$
Hence, $S\in Q_2$, and the claim follows.
\end{proof}

\subsection{Probabilistic lemmas}

In this short subsection, we derive probabilistic counterparts of the last two lemmas.

\begin{lemma}\label{l: aux pqhgaf}
Let $r\geq 2$ and let $X$ be a random vector in $\R^n$
satisfying condition \eqref{eq: density r-condition} for $m=r$ and some $K>0$.
Then for any numbers $a,\tau>0$ and any fixed subspaces $H_1\subset H_2\subset\R^n$ of dimensions
$n-r$ and $n-1$, respectively, we have
$$\Prob\big\{\dist(X,H_1)\leq \tau\;\;\mbox{and}\;\;\dist(X,H_2)\leq a\big\}
\leq\frac{\pi^{(r-1)/2}a K\tau^{r-1}}{\Gamma(\frac{r+1}{2})}.$$
\end{lemma}
\begin{proof}
Clearly, probability of the event is bounded by $K$ times the Lebesgue volume of the slab
$$\big\{x\in\R^r:\, \|x\|_2\leq\tau\;\;\mbox{and}\;\;\dist(x,\spn\{e_1,\dots,e_{r-1}\})\leq a\big\}.$$
A trivial computation gives the result.
\end{proof}

\begin{lemma}
Let $r\geq 2$ and let $X_1,X_2,\dots,X_r$ be independent isotropic random vectors in $\R^n$ so that
$X_1$ satisfies condition \eqref{eq: density r-condition} with $m=1$ and some $K>0$.
Then for any numbers $a,h>0$ and any fixed subspace $H\subset\R^n$ of dimension $n-r$ we have
$$\Prob\big\{\dist(X_1,\spn\{H,X_2,\dots,X_r\})\leq a\;\mbox{and}\;\exists i\geq 2,\,
\dist(X_i,H)\geq h\big\}\leq \frac{aKr^3}{h^2}.$$
\end{lemma}
\begin{proof}
Let $v_1,v_2,\dots,v_r$ be a fixed orthonormal basis of $H^\perp$.
Note that, in view of the moment assumptions on vectors $X_2,\dots,X_r$,
for any $i\geq 2$ we have, by the union bound,
$$\Prob\{\dist(X_i,H)\geq h\}\leq \sum_{j=1}^r\Prob\Big\{|\langle X_i,v_j\rangle|\geq \frac{h}{\sqrt{r}}\Big\}
\leq \frac{r^2}{h^2},$$
whence
$$\Prob\big\{\exists i\geq 2,\,\dist(X_i,H)\geq h\big\}\leq \frac{r^3}{h^2}.$$
It remains to note that, conditioned on any realization of $X_2,\dots,X_r$,
the probability that $\dist(X_1,\spn\{H,X_2,\dots,X_r\})\leq a$ is bounded by $aK$.
\end{proof}

As elementary corollaries of the last two lemmas and Markov's inequality, we get

\begin{lemma}\label{l: prob Q_1}
Let $r\geq 2$, 
and let $A$ be $n\times n$ random matrix with independent rows,
satisfying condition \eqref{eq: density r-condition}
for $m=r$ and some $K>0$.
Fix parameters $\tau,a>0$ and define a random subset $\mathcal Q_1$ of $\tuples_{n,r}$
as
$$\mathcal Q_1:=\big\{S\in \tuples_{n,r}:\,\exists j\in S\mbox{ s.t.\ }\dist(R_j(A),H^S(A))\leq \tau\;\;
\mbox{and}\;\;\dist(R_j(A),H^j(A))\leq a\big\}.$$
Then
$$\Exp\big|\mathcal Q_1\big|\leq \frac{\pi^{(r-1)/2}a Kr\tau^{r-1}}{\Gamma(\frac{r+1}{2})}{n\choose r}.$$
\end{lemma}

\begin{lemma}\label{l: prob Q_2}
Let $A$ be $n\times n$ matrix with independent isotropic rows
satisfying condition \eqref{eq: density r-condition}
with $m=1$ and some $K>0$.
Fix parameters $r\geq 2$ and $a,h>0$ and define
a random subset $\mathcal Q_2$ of $\tuples_{n,r}$ as
$$\mathcal Q_2:=\big\{S\in \tuples_{n,r}:\,\exists i,j\in S,\,i\neq j,\;\dist(R_i(A),H^i(A))\leq a;\;
\dist(R_j(A),H^S(A))\geq h\big\}.$$
Then
$$\Exp\big|\mathcal Q_2\big|\leq \frac{aKr^4}{h^2}{n\choose r}.$$
\end{lemma}

\subsection{Proof of Theorem~\ref{th:revisiting}}

\begin{prop}\label{p: no shift}
Let $n/4\geq r\geq 2$, and
let $A$ be an $n\times n$ matrix with independent isotropic rows
satisfying condition
\eqref{eq: density r-condition}
for $m=1$, $m=r$ and some $K>0$.
Then for every number $a>0$ and every integer $k\leq n$ we have
$$\Prob\big\{\exists I\subset [n]:\,|I|\geq k,\;\dist(R_i(A),H^i(A))\leq a\;\mbox{ for all }i\in I\big\}
\leq w\, aK\bigg(\frac{n}{k}\bigg)^{\frac{r+1}{3r-1}},$$
where $w=w(r)>0$ depends only on $r$.
\end{prop}
\begin{proof}
Let us denote
$$\Event:=\big\{\exists I\subset [n]:\,|I|\geq k,\;\dist(R_i(A),H^i(A))\leq a\;\mbox{ for all }i\in I\big\},$$
and set $\beta:=\Prob(\Event)$ and $b:=\frac{\beta}{8K}$.
First, note that if $a\geq b$ then necessarily $\beta\leq 8aK$, and there is nothing to prove.
In what follows, we assume that $a<b$.
Further, note that for any $i\leq n$ the probability of the event $\Event_i:=\{\dist(R_i(A),H^i(A))\leq b\}$
is bounded from above by $2bK=\frac{\beta}{4}$. Hence, if $\chi_{\Event_i}$ is the indicator function of 
$\Event_i$ ($i\leq n$) then, from Markov's inequality,
$$\Prob\Big\{\sum_{i=1}^n \chi_{\Event_i}\geq \frac{n}{2}\Big\}\leq \frac{\beta}{2}.$$
Thus, setting
$$\widetilde\Event:=\Event\cap
\Big\{\exists J\subset [n]:\,|J|\geq \frac{n}{2},\;\dist(R_i(A),H^i(A))\geq b\;\mbox{ for all }i\in J\Big\},$$
we get $\Prob(\widetilde\Event)\geq\frac{\beta}{2}$.

Define random sets $\mathcal Q_1$ and $\mathcal Q_2$ as in Lemmas~\ref{l: prob Q_1} and~\ref{l: prob Q_2},
where we take $h:=\frac{\tau b}{2ar}$ and
$$\tau:=\bigg(\frac{aK}{\beta}\bigg)^{2/(r+1)}.$$
The two lemmas and Markov's inequality imply
\begin{equation}\label{eq: aux 9t8h4}
\Prob\Big\{|\mathcal Q_1\big|\geq \frac{k}{2}{\lceil n/2\rceil \choose r-1}\Big\}
\leq \frac{2\pi^{(r-1)/2}a Kr\tau^{r-1}}{k\Gamma(\frac{r+1}{2}){\lceil n/2\rceil \choose r-1}}{n\choose r}
\leq C^r aK\tau^{r-1}\frac{n}{k},
\end{equation}
and
\begin{equation}\label{eq: aux -98thke}
\Prob\Big\{|\mathcal Q_2\big|\geq \frac{k}{2}{\lceil n/2\rceil \choose r-1}\Big\}
\leq \frac{2 \frac{aKr^4}{h^2}{n\choose r}}{k{\lceil n/2\rceil \choose r-1}}
\leq C^r \frac{aK}{h^2}\, \frac{n}{k},
\end{equation}
where $C>0$ is a sufficiently large universal constant.

Now, we relate the sets $\mathcal Q_1$ and $\mathcal Q_2$
to the non-random sets $Q_1$ and $Q_2$ from Lemma~\ref{l: Q1Q2}.
Take any point $\omega\in \widetilde\Event$, and set
\begin{align*}
\widetilde I(\omega)&:=\big\{i\leq n:\,\dist(R_i(A(\omega)),H^i(A(\omega)))\leq a\big\};\\
\widetilde J(\omega)&:=\big\{i\leq n:\,\dist(R_i(A(\omega)),H^i(A(\omega)))\geq b\big\}.
\end{align*}
Note that $\widetilde I(\omega),\widetilde J(\omega)$
are two disjoint subsets of $[n]$ with $|\widetilde J(\omega)|\geq n/2$, $|\widetilde I(\omega)|\geq k$,
and, with $Q_1$ and $Q_2$ defined as in Lemma~\ref{l: Q1Q2} (with $\widetilde I(\omega)$,
$\widetilde J(\omega)$ replacing $I$ and $J$), we have
$$Q_1(A(\omega))=\mathcal Q_1(\omega)\quad\mbox{and}\quad Q_2(A(\omega))\subset\mathcal Q_2(\omega).$$
Further, in view of Lemma~\ref{l: Q1Q2}, we have
$$|Q_1(A(\omega))|+|Q_2(A(\omega))|\geq k{\lceil n/2\rceil\choose r-1}.$$
Hence, we have
$$|\mathcal Q_1(\omega)|+|\mathcal Q_2(\omega)|\geq k{\lceil n/2\rceil\choose r-1}\;\;\mbox{for all }\omega\in\widetilde\Event.$$
Together with the bound $\Prob(\widetilde \Event)\geq\frac{\beta}{2}$, as well as \eqref{eq: aux 9t8h4} and~\eqref{eq: aux -98thke},
this leads to the inequality
$$\frac{\beta}{2}\leq C^r aK\tau^{r-1}\frac{n}{k}+C^r \frac{aK}{h^2}\,\frac{n}{k}.$$
Simplifying, we get
$$\beta\leq v\, aK\frac{n}{k}\bigg(\frac{aK}{\beta}\bigg)^{2(r-1)/(r+1)},$$
whence
$$\beta\leq w\, aK\bigg(\frac{n}{k}\bigg)^{\frac{r+1}{3r-1}}$$
for some $v,w>0$ depending only on $r$.
\end{proof}

\bigskip

\begin{proof}[Proof of Theorem~\ref{th:revisiting}]
We will assume that $n$ is sufficiently large.
Let $A$ be an $n\times n$ random matrix with independent rows $R_1(A),R_2(A),\dots,R_n(A)$ satisfying conditions of the theorem.
The columns $\col_1(A^{-1}),\col_2(A^{-1}),\dots,\col_n(A^{-1})$ of the inverse matrix $A^{-1}$
and rows $R_1(A),R_2(A),\dots,R_n(A)$
of $A$ form a biorthogonal system in $\R^n$, and, in particular, $\|\col_i(A^{-1})\|_2=\frac{1}{\dist(R_i(A),H^i(A))}$, $i\leq n$.
Thus, the Hilbert--Schmidt norm of $A^{-1}$ can be expressed as
$$\|A^{-1}\|_{HS}^2=\sum_{i=1}^n \dist(R_i(A),H^i(A))^{-2}.$$
Hence, we need to show that
$$\Prob\Big\{\sum\nolimits_{i=1}^n \dist(R_i(A),H^i(A))^{-2}\geq t^2 n\Big\}\leq \frac{C K}{t},\quad t>0,$$
for a universal constant $C>0$.
Set $p_0:=21/10$, and fix any $t\geq 1$. For every number $\ell\in\{0,1,\dots,\lfloor \log_2(n)\rfloor\}$ let 
$$D_\ell:=\widetilde C t^{-1}\bigg(\frac{2^\ell}{n}\bigg)^{1/p_0},$$
and let
$\Event_\ell$
be the event
$$\Big\{\exists I_\ell\subset[n]:\;|I_\ell|\geq 2^\ell\;\mbox{ and }\;\dist(R_i(A),H^i(A))
\leq D_\ell\;\forall\,i\in I_\ell\Big\},$$
where $\widetilde C>0$ is a large universal constant to be chosen later.
Note that on the complement of $\bigcup_{\ell=0}^{\lfloor \log_2(n)\rfloor}\Event_\ell$
we have $\dist(R_i(A),H^i(A))\geq D_0$ for all $i\leq n$ (since we are in the complement of $\Event_0$), and
$|\{i\leq n:\,\dist(R_i(A),H^i(A))\leq D_\ell\}|< 2^\ell$ for all $0\leq\ell\leq \lfloor \log_2(n)\rfloor$. Hence,
everywhere on $\big(\bigcup_{\ell=0}^{\lfloor \log_2(n)\rfloor}\Event_\ell\big)^c$,
\begin{align*}
\sum_{i=1}^n &\dist(R_i(A),H^i(A))^{-2}\\
&\leq
2\sum_{i=1}^n\sum_{\ell=-\infty}^\infty D_\ell^{-2}\chi_{\{\dist(R_i(A),H^i(A))\leq D_\ell\}}\\
&=2\sum_{\ell=0}^\infty D_\ell^{-2}\big|\big\{i\leq n:\,\dist(R_i(A),H^i(A))\leq D_\ell\big\}\big|\\
&\leq 2\sum_{\ell=0}^{\lfloor \log_2(n)\rfloor} D_\ell^{-2} \,2^{\ell}
+2n\sum_{\ell=\lfloor \log_2(n)\rfloor+1}^\infty D_\ell^{-2}\\
&=2\sum_{\ell=0}^{\lfloor \log_2(n)\rfloor} {\widetilde C}^{-2} t^{2}n^{2/p_0} 2^{\ell-2\ell/p_0}
+2n\sum_{\ell=\lfloor \log_2(n)\rfloor+1}^\infty \widetilde C^{-2} t^{2}\bigg(\frac{n}{2^\ell}\bigg)^{2/p_0}\\
&< t^{2}n,
\end{align*}
where the last relation holds as long as $\widetilde C$ is taken sufficiently large.
Thus, we have inclusion
$$\Big\{\sum\nolimits_{i=1}^n \dist(R_i(A),H^i(A))^{-2}\geq t^2 n\Big\}
\subset \bigcup_{\ell=0}^{\lfloor \log_2(n)\rfloor}\Event_\ell.$$
It remains to note that, in view of Proposition~\ref{p: no shift}, we have
$$\Prob\Big(\bigcup_{\ell=0}^{\lfloor \log_2(n)\rfloor}\Event_\ell\Big)
\leq \sum_{\ell=0}^{\lfloor \log_2(n)\rfloor}
C'\,D_\ell
K\bigg(\frac{n}{2^\ell}\bigg)^{\frac{5}{11}}\leq \frac{CK}{t}$$
for some universal constants $C,C'>0$.
\end{proof}


\section{Non-centered matrices}\label{s:shifted}

In this section, we will prove the main result of the paper, dealing with shifted random matrices.
Issues with a covering argument, discussed in context of centered matrices, create fundamental problems
in this new setting. Indeed, the norm of the shifted matrix $A+M$ can be arbitrarily large,
and covering of the random ellipsoid $(A+M)(B_2^n)$ with translates of $C\sqrt{n}B_2^n$ can
have enormous cardinality. Even if possible, an ``optimal'' treatment of non-centered matrices via a covering argument
should be much more delicate, taking into account that the matrix $M$ can have different ``magnitudes'' in different directions.
We would like to mention that estimates for the smallest singular value of non-Gaussian matrices, which depend on the norm of $M$,
have been previously obtained, see, in particular, paper \cite{TV-cn} by Tao and Vu. However, those estimates
are suboptimal in the class of ``sufficiently smooth'' distributions treated in our paper.

At the same time, the approach we developed in the previous section, without
significant modifications, cannot be applied in the non-centered setting as the basic conceptual element
of the above argument --- approximately the same length
of columns of the inverse matrix w.h.p.\ --- is clearly no longer valid in the shifted case.
In particular, if we define the random set $\mathcal Q_2$ the same way as in Lemma~\ref{l: prob Q_2},
but with shifted matrix $A+M$, then there is no reason for expectation $\Exp |\mathcal Q_2|$
to be small for large values of parameter $h$.
At the same time, the idea of estimating correlations of pairs of the columns (now, in a more complex situation)
will play a crucial role here.

Let us start with the observation mentioned in the introduction, which
shows that ``shift--independent'' small ball probability estimates are impossible in the discrete setting.
A slightly weaker version of the statement below is given in paper \cite{TV-cn} by Tao and Vu.
\begin{observation}\label{p:shifted Bernoulli}
For any sufficiently large $n$ and any $\tau\geq n$ there is a fixed $n\times n$ matrix $M$ with
$s_{\max}(M)=\tau$ with the following property.
Let $B=(b_{ij})$ be the $n\times n$ random Bernoulli matrix with i.i.d.\ $\pm 1$ entries. Then
$$\Prob\{s_{\min}(B+M)\leq Cn/\tau\}\geq 1/5,$$
where $C>0$ is a universal constant. In particular, with probability at least $0.19$ we have
$$\kappa(B+M)\geq \frac{c'\,\tau^2}{n}=\frac{c'\,s_{\max}(M)^2}{n}.$$
\end{observation}
\begin{proof}
Let $M=(m_{ij})$ be diagonal, with $m_{ii}=\tau$ for $i\leq n-2$ and $m_{n-1,n-1}=m_{nn}=0$.
Define a random $n$-dimensional
vector $X$ by
$$X_i:=-\frac{b_{i,n-1}+b_{in}}{\tau},\;\;i=1,\dots,n-2;\;\;\;\;X_{n-1}=X_n=1.$$
Observe that, in view of the conditions on $\tau$, the vector $X$ satisfies
$\sqrt{2}\leq \|X\|_2<2$ deterministically.
Further, let us estimate the Euclidean norm of $(B+M)X$.
Let us define an auxiliary random matrix $B'$ which has the same first $n-2$ columns as $B$, but
the last two columns of $B'$ are zeros. Then, by the construction of the vector $X$, we have
$$\|(B+M-B')X\|_2^2=(b_{n-1,n-1}+b_{n-1,n})^2+(b_{n,n-1}+b_{nn})^2.$$
This quantity clearly equals zero with probability $1/4$, and depends only on the four entries of $B$ in the bottom right corner.
Further,
$$\|B'X\|_2\lesssim s_{\max}(B')\,\frac{\sqrt{n}}{\tau}\lesssim \frac{n}{\tau},$$
where the last inequality holds with probability at least $0.99$ (assuming a sufficiently large implied constant multiple).
Combining the two estimates for $\|(B+M-B')X\|_2$ and $\|B'X\|_2$,
we get for some universal constant $C'>0$:
$$\Prob\big\{\|(B+M)X\|_2\leq C' n/\tau\big\}\geq 1/5.$$
It remains to use the estimates of the norm of $X$.
\end{proof} 

In case of bounded density of the entries, a simple analysis gives a much better upper estimate of the condition
number than is available in the Bernoulli setting. The following fact was mentioned without proof
in \cite[Section~7.1]{SST}; see also \cite[Section~4.4]{BC12}.
\begin{observation}\label{o: bounded density trivial}
Let $\xi$ be a real valued random variable with the distribution density bounded above by $K$
and let $A$ be an $n\times n$ random matrix whose entries are i.i.d.\ random variables equidistributed with
$\xi$. Further, let $M$ be any $n\times n$ deterministic matrix. Then
$$\Prob\big\{\|(A+M)^{-1}\|_{HS}\geq t n\big\}\leq \frac{C\,K}{t},\quad t>0,$$
where $C>0$ is a universal constant.
\end{observation}
\begin{proof}
The columns of $(A+M)^{-1}$ and the rows
of $A+M$ form a biorthogonal system in $\R^n$, and, in particular, $\|\col_i((A+M)^{-1})\|_2=\frac{1}{\dist(R_i(A+M),H^i(A+M))}$, $i\leq n$.
Applying Theorem~1.2 from \cite{RV_images}, we get that for any $i\leq n$, conditioned on any realization of $H^i(A+M)$
(so that $\dim H^i(A+M)=n-1$),
the conditional distribution density of the distance $\dist(R_i(A+M),H^i(A+M))$ is bounded from above by $2\sqrt{2}K$.
It follows that for every $s>0$ and $i\leq n$ we have
$$\Prob\big\{\|\col_i((A+M)^{-1})\|_2^2\geq s\big\}\leq \frac{2\sqrt{2} K}{\sqrt{s}}.$$
Now, fix any $t\geq 1$. We can write
\begin{align*}
\Prob\big\{\|(A+M)^{-1}\|_{HS}\geq t n\big\}&=\Prob\Big\{\sum_{i=1}^n \|\col_i((A+M)^{-1})\|_2^2\geq t^2 n^2\Big\}\\
&\leq \sum_{k=1}^\infty\Prob\Big\{\big|\big\{i\leq n: \|\col_i((A+M)^{-1})\|_2^2\geq 2^{-\frac{3k}{2}}t^2 n^2\big\}\big|\geq c\,2^{k}\Big\},
\end{align*}
where $c>0$ is a sufficiently small universal constant. For each $k\geq 1$, applying Markov's inequality
to the sum of indicators of the events $\{\|\col_i((A+M)^{-1})\|_2^2\geq 2^{-\frac{3k}{2}}t^2 n^2\}$, we get
\begin{align*}
\Prob\Big\{&\big|\big\{i\leq n:\, \|\col_i((A+M)^{-1})\|_2^2\geq 2^{-\frac{3k}{2}}t^2 n^2\big\}\big|\geq c\,2^{k}\Big\}\\
&\leq \frac{\sum_{i=1}^n \Prob\{\|\col_i((A+M)^{-1})\|_2^2\geq 2^{-\frac{3k}{2}}t^2 n^2\}}{c\,2^{k}}\\
&\leq\frac{C'\,K}{t\,2^{k/4}}.
\end{align*}
Summing up over $k\geq 1$, we get the result.
\end{proof}

Observation~\ref{o: bounded density trivial} is not optimal, with an error factor of order $\sqrt{n}$
for the probability bounds. Below, we develop a more delicate argument.

\subsection{An example of event partitioning}\label{subs:event}

In comparison with the proofs in the first part of the paper,
the main idea in the non-centered setting is to find a useful partition of the
probability space and the event in question.
The techniques we discuss here are more powerful, although heavier.
Instead of directly defining the key object of the section --- the $(\alpha,\eta)$--structure
on the probability space --- let us make an intermediate step by considering once again
the setting of centered matrices, but this time recasting the argument in the form of event partitions.
We will present a rough form of the argument, with excessive logarithmic multiples.
Yet, we hope that this will serve as a link between the two different approaches
studied in this work.

Consider a random $n\times n$ matrix $A$ with independent isotropic rows with subgaussian two-dimensional
marginals having uniformly bounded distribution densities.
It will be convenient to view the matrix $A$ as defined on a product probability space
$(\Omega,\Prob)=\prod_{i=1}^n (\Omega_i,\Prob_i)$, where the $i$-th component
corresponds to the $i$-th row, $i\leq n$. Whenever there is
necessity of specifying the point of the probability space where
the matrix rows are evaluated, we will abbreviate $R_i(A(\omega_1,\dots,\omega_n))$
for $R_i(\omega_i)$, $i\leq n$.
Let $\Event$ be the event
$$\Event:=\big\{\|A^{-1}\|_{HS}\in [t\sqrt{n},2t\sqrt{n}]\big\}$$
for some $t\geq 1$. In this argument (which is provided only as an illustration), we also assume that $t\leq e^{c\log^2 n}$
for a sufficiently small $c>0$.
Our goal would be to show that $\Prob(\Event)\leq C/ t$. In fact, to avoid as many technical details
as possible, we will obtain a weaker estimate $\Prob(\Event)\leq C\,\log^6 n\,/t$.

Define an index set
$$
\Lambda:=\big\{(\lambda_1,\lambda_2):\;\lambda_1,\lambda_2\in
\{-\lfloor\log_2^2 n\rfloor,\dots,0\}\big\}\cup\{\infty_1,\infty_2\},
$$
where ``$\infty_1,\infty_2$'' will denote indices corresponding to ``tail'' events of very small probabilities.
For any $i\leq n$, we define a partitioning $(\Event_{i,\lambda})_{\lambda\in \Lambda}$ of the event $\Event$ as follows:
\begin{align*}
&\Event_{i,\infty_1}:=\Event\cap\big\{\dist(R_i(A),H^{i}(A))<2^{-\lfloor\log_2^2 n\rfloor}\log_2^2n
\quad\mbox{or}\quad \dist(R_i(A),H^{\{i,j\}}(A))\geq 2\log_2^2n\\
&\hspace{2.6cm}\mbox{for some }j\neq i\big\};\\
&\Event_{i,\lambda}:=(\Event\setminus\Event_{i,\infty_1})\cap\big\{\dist(R_i(A),H^{i}(A))\in
[2^{\lambda_1}\log_2^2n,2^{\lambda_1+1}\log_2^2n)\;\;\mbox{and}\\
&\hspace{2.2cm}\lfloor n/2\rfloor\kmax(\{\dist(R_j(A),H^{j}(A)):\;j\neq i\})\in
[2^{\lambda_2}\log_2^2n,2^{\lambda_2+1}\log_2^2n)\big\},\\
&\hspace{2cm}\lambda=(\lambda_1,\lambda_2)\in\Lambda\setminus\{\infty_1,\infty_2\};\\
&\Event_{i,\infty_2}:=
\Event\,\setminus \bigcup\limits_{\lambda\in\Lambda\setminus\{\infty_2\}}\Event_{i,\lambda}.
\end{align*}
Here, ``$\lfloor n/2\rfloor\kmax\, T$'' denotes the $\lfloor n/2\rfloor$-th largest element of a multiset $T$.
We will estimate the probability of the event $\Event$ by studying the structure of the partitions
$(\Event_{i,\lambda})_{\lambda\in\Lambda}$.

Note that for any $\lambda\in\Lambda$ and $i\leq n$, in view of Tonelli's theorem,
the function
$$(\omega_1,\dots,\omega_n) \longrightarrow
\Prob_i\{\widetilde \omega_i\in\Omega_i:\,(\omega_1,\dots,\omega_{i-1},\widetilde\omega_i,\omega_{i+1},\dots,\omega_n)
\in\Event_{i,\lambda}\}^{-1}$$
is measurable, and
\begin{align*}
&\int\limits_{\Event_{i,\lambda}}
\frac{d\,\Prob_1(\omega_1)\,d\,\Prob_2(\omega_2)\dots d\,\Prob_n(\omega_n)}
{\Prob_i\{\widetilde \omega_i\in\Omega_i:\,(\omega_1,\dots,\omega_{i-1},\widetilde\omega_i,\omega_{i+1},\dots,\omega_n)
\in\Event_{i,\lambda}\}}\\
&\hspace{1cm}=
\int\limits_{\prod_{j\neq i}\Omega_j}\,
\int\limits_{\{\omega_i\in\Omega_i:\,
(\omega_1,\dots,\omega_n)
\in\Event_{i,\lambda}\}}
\frac{d\,\Prob_1(\omega_1)\,d\,\Prob_2(\omega_2)\dots d\,\Prob_n(\omega_n)}
{\Prob_i\{\widetilde \omega_i\in\Omega_i:\,(\omega_1,\dots,\omega_{i-1},\widetilde\omega_i,\omega_{i+1},\dots,\omega_n)
\in\Event_{i,\lambda}\}}\\
&\hspace{1cm}\leq \int\limits_{\prod_{j\neq i}\Omega_j}\,d\,\Prob_1(\omega_1)\,d\,\Prob_2(\omega_2)\dots 
\,d\,\Prob_{i-1}(\omega_{i-1})\,d\,\Prob_{i+1}(\omega_{i+1})\dots\,d\,\Prob_n(\omega_n)=1.
\end{align*}

Given any $i\leq n$ and a point $\omega=(\omega_1,\dots,\omega_n)\in\Omega$, let
$\lambda(i,\omega)\in\Lambda$ be the index such that $\omega\in\Event_{i,\lambda(i,\omega)}$.
Then, summing up the above inequality over $i$ and $\lambda\in\Lambda$, we get
\begin{equation}\label{eq: aux 09827520560298}
\int\limits_{\Event}\sum_{i=1}^n
\frac{d\,\Prob_1(\omega_1)\,d\,\Prob_2(\omega_2)\dots d\,\Prob_n(\omega_n)}
{\Prob_i\{\widetilde \omega_i\in\Omega_i:\,(\omega_1,\dots,\omega_{i-1},\widetilde\omega_i,\dots,\omega_n)
\in\Event_{i,\lambda(i,\omega)}\}}\leq \sum\limits_{\lambda\in\Lambda}\sum\limits_{i=1}^n
1= |\Lambda|\,n.
\end{equation}
Thus, if we show for some $\gamma>0$ that
$
\sum_{i=1}^n\frac{1}
{\Prob_i\{\widetilde \omega_i\in\Omega_i:\,(\omega_1,\dots,\omega_{i-1},\widetilde\omega_i,\dots,\omega_n)
\in\Event_{i,\lambda(i,\omega)}\}}\geq\gamma
$
everywhere on $\Event$ then immediately we will get $\Prob(\Event)\leq |\Lambda|\,n/\gamma$.

\noindent {\bf Claim.} {\it{}We have
$$
\sum_{i=1}^n
\frac{1}
{\Prob_i\{\widetilde \omega_i\in\Omega_i:\,(\omega_1,\dots,\omega_{i-1},\widetilde\omega_i,\omega_{i+1},\dots,\omega_n)
\in\Event_{i,\lambda(i,\omega)}\}}\geq \min\bigg(\frac{c' tn}{\log_2^2n},e^{c' \log_2^2n}\bigg)
$$
for all $\omega\in\Event$, where $c'>0$ may only depend on the subgaussian moment.}

\medskip

\noindent {\bf Proof of claim.}
Fix any $\omega=(\omega_1,\dots,\omega_n)\in\Event$.

{\bf(1)} First, assume that there is $i\leq n$ such that $\omega\in\Event_{i,\infty_1}$.
In view of the definition of $\Event_{i,\infty_1}$, this immediately implies that 
\begin{align*}
&\sum_{j=1}^n\frac{1}
{\Prob_j\{\widetilde \omega_j\in\Omega_j:\,(\omega_1,\dots,\omega_{j-1},\widetilde\omega_j,\omega_{j+1},\dots,\omega_n)
\in\Event_{j,\lambda(j,\omega)}\}}\\
&\geq\frac{1}
{\Prob_i\{\widetilde \omega_i\in\Omega_i:\,(\omega_1,\dots,\omega_{i-1},\widetilde\omega_i,\omega_{i+1},\dots,\omega_n)
\in\Event_{i,\infty_1}\}}\\
&\geq \inf\limits_{E_1\subset E_2\subset\R^n}
\big(n^{-1}\,\Prob_i\{\widetilde \omega_i\in\Omega_i:\,\dist(R_i(\widetilde\omega_i),E_2)
\leq 2^{-\lfloor \log_2^2n\rfloor}\log_2^2n\mbox{ or}\\
&\hspace{3.7cm}\dist(R_i(\widetilde\omega_i),E_1)\geq 2\log_2^2n\}^{-1}\big),
\end{align*}
where the infimum is taken over all non-random subspaces $E_1\subset E_2\subset\R^n$
of dimensions $n-2$ and $n-1$, respectively.
Applying concentration/anticoncentration inequalities for the distance, we then obtain
$$
\sum_{j=1}^n\frac{1}
{\Prob_j\{\widetilde \omega_j\in\Omega_j:\,(\omega_1,\dots,\omega_{j-1},\widetilde\omega_j,\omega_{j+1},\dots,\omega_n)
\in\Event_{j,\lambda(j,\omega)}\}}\geq e^{c' \log_2^2n}
$$
for some constant $c'>0$ which may only depend on the subgaussian moment.

{\bf(2)} Treatment of the case $\omega\in\Event_{i,\infty_2}$ for some $i\leq n$ is immediately reduced to the above
since the condition $\omega\in\Event_{i,\infty_2}$ implies that 
$$
\lfloor n/2\rfloor\kmax(\{\dist(R_j(A(\omega)),H^{j}(A(\omega))):\;j\neq i\})\in
(0,2^{-\lfloor \log_2^2n\rfloor}\log_2^2n)\cup [2\log_2^2n,\infty),
$$
whence there is index $j\in[n]$ such that $\omega\in\Event_{j,\infty_1}$.

{\bf(3)} Next, assume that the first
case does not hold, but there is $i\leq n$ such that $\omega=(\omega_1,\dots,\omega_n)\in\Event_{i,(\lambda_1,\lambda_2)}$
for some $\lambda_2\leq \log_2(1/t)$.
By the definition of event $\Event_{i,(\lambda_1,\lambda_2)}$, this implies that for at least $n/2$ indices $j\in[n]$
we have
$\dist(R_j(A(\omega)),H^{j}(A(\omega)))< 2^{\lambda_2+1}\log_2^2n$,
whence $\omega\in \Event_{j,(\lambda_{1,j},\lambda_{2,j})}$ for some
$\lambda_{1,j}\leq \lambda_2$ and for some $\lambda_{2,j}$.
Denote this set of indices $j$ by $J$.
Combining individual probability bounds, we then get
\begin{align*}
&\sum_{j=1}^n\frac{1}
{\Prob_j\{\widetilde \omega_j\in\Omega_j:\,(\omega_1,\dots,\omega_{j-1},\widetilde\omega_j,\omega_{j+1},\dots,\omega_n)
\in\Event_{j,\lambda(j,\omega)}\}}\\
&\geq 
\sum_{j\in J}\frac{1}
{\Prob_j\{\widetilde \omega_j\in\Omega_j:\,(\omega_1,\dots,\omega_{j-1},\widetilde\omega_j,\omega_{j+1},\dots,\omega_n)
\in\Event_{j,(\lambda_{1,j},\lambda_{2,j})}\}}
\geq \frac{c' tn}{\log_2^2n},
\end{align*}
where the last inequality is implied by the definition
of $\Event_{j,(\lambda_{1,j},\lambda_{2,j})}$, the condition
$\lambda_{1,j}\leq \log_2(1/t)$, and the property that for any fixed $(n-1)$--dimensional subspace $E$,
the variable $\dist(R_j(\widetilde\omega),E)$ has a uniformly bounded distribution density.

{\bf(4)} Finally, suppose that none of the three cases considered above holds.
Take any $i\leq n$, and set
$(\lambda_{1,i},\lambda_{2,i}):=\lambda(i,\omega)$.
Note that for any $j\neq i$, the condition 
$$
\dist(R_i(A),H^{i}(A))\in [2^{\lambda_{1,i}}\log_2^2 n,2^{\lambda_{1,i}+1}\log_2^2 n)\quad\mbox{ and }\quad
\dist(R_j(A),H^{j}(A))\geq 2^{\lambda_{2,i}}\log_2^2 n
$$
deterministically implies, in view of Lemma~\ref{l: aux9876}, that
$$
\dist(R_j(A),H^{i,j}(A))\geq 2^{\lambda_{2,i}-\lambda_{1,i}-3}\dist(R_i(A),H^{i,j}(A)).
$$
Hence, applying Lemma~\ref{l: aux pqhgaf},
we can get for any $j\neq i$:
\begin{align*}
\Prob_i\big\{&\widetilde\omega_i\in\Omega_i:\,
\dist(R_i(\widetilde\omega_i),H^{i}(A(\omega)))
\in [2^{\lambda_{1,i}}\log_2^2 n,2^{\lambda_{1,i}+1}\log_2^2 n)\\
&\mbox{and }\;\dist(R_j(\omega_j),H^{j}(A(\omega_1,\dots,\omega_{i-1},\widetilde\omega_i,\omega_{i+1},\dots,\omega_n)
))\geq 2^{\lambda_{2,i}}\log_2^2 n\big\}\\
&\leq C \,2^{2\lambda_{1,i}-\lambda_{2,i}}\,\dist(R_j(\omega_j),H^{i,j}(A(\omega)))\, \log_2^2 n
\leq C' \,2^{2\lambda_{1,i}-\lambda_{2,i}}\, \log_2^4 n,
\end{align*}
where in the last inequality we used the condition that $\omega\notin \Event_{j,\infty_1}$ for all $j$.
It can be checked, by applying Markov's inequality, that the last relation implies
\begin{align*}
\Prob_i\{&\widetilde \omega_i\in\Omega_i:\,(\omega_1,\dots,\omega_{i-1},\widetilde\omega_i,\omega_{i+1},\dots,\omega_n)
\in\Event_{i,(\lambda_{1,i},\lambda_{2,i})}\}\\
&\leq C''\, 2^{2\lambda_{1,i}-\lambda_{2,i}}\, \log_2^4 n\\
&\leq C'' \dist(R_i(\omega_i),H^{i}(A(\omega)))^2\,t,
\end{align*}
where we have used that $\lambda_{2,i}\geq \log_2(1/t)$.
Hence, summing up over all indices $i$, we get
\begin{align*}
\sum_{i=1}^n
&\frac{1}
{\Prob_i\{\widetilde \omega_i\in\Omega_i:\,(\omega_1,\dots,\omega_{i-1},\widetilde\omega_i,\omega_{i+1},\dots,\omega_n)
\in\Event_{i,\lambda(i,\omega)}\}}\\
&\geq \frac{\widetilde c}{t}\,\sum_{i=1}^n\dist(R_i(A(\omega)),H^{i}(A(\omega)))^{-2}
=\frac{\widetilde c}{t}\|A^{-1}(\omega)\|_{HS}^2
\geq \widetilde c t n,
\end{align*}
where we have used the negative second moment identity and the definition of $\Event$.

Combining all the cases, we get that for any point $\omega\in\Event$, we have
$$
\sum_{i=1}^n
\frac{1}
{\Prob_i\{\widetilde \omega_i\in\Omega_i:\,(\omega_1,\dots,\omega_{i-1},\widetilde\omega_i,\omega_{i+1},\dots,\omega_n)
\in\Event_{i,\lambda(i,\omega)}\}}\geq \min\bigg(\frac{c' tn}{\log_2^2n},e^{c' \log_2^2n}\bigg),
$$
proving the claim.

\medskip

Hence, applying \eqref{eq: aux 09827520560298}, we get
$$
\Prob(\Event)\leq \frac{C\log_2^6 n}{t},
$$
as long as $1\leq t\leq e^{c\log_2^2 n}$ for a sufficiently small constant $c>0$.

\medskip

The proof sketch
presented above, in its essence, is close to the argument from Section~\ref{s:revisited}.
The main two components --- concentration/anticoncentration estimates for the distances
and averaging arguments --- appear in both settings.
However, in the approach presented here, applications of Markov's inequality
are {\it partially} replaced with the use of relation \eqref{eq: aux 09827520560298}, which
serves the same purpose conceptually. 
The structure of the argument becomes heavier: we had to consider several different cases
for our point $\omega\in\Event$, although the argument was already simplified by
restricting the range of $t$ and by making very strong assumptions on the distribution.

The advantage of the approach given here is that it can be adapted to treat non-centered random matrices,
by ``re-weighing'' the relation \eqref{eq: aux 09827520560298}.
Let us consider a simpest setting: we would like to estimate the Hilbert--Schmidt norm of $(A+M)^{-1}$,
where $M$ is a fixed $n\times n$ matrix such that the first $\delta n$ of its rows are zeros (with $\delta\in(0,1)$ being
some parameter), and
the last $n-\delta n$ are pairwise orthogonal and having large Euclidean norms (say, larger than $e^{C\log_2^2 n}$
for a big constant $C$). If we define the events $\Event_{i,\lambda}$
and attempt to argue exactly the same way as in the above discussion, we will not be able to obtain satisfactory estimates.
A way to correct the argument would be to replace
$\lfloor n/2\rfloor\kmax(\{\dist(R_j(A),H^{j}(A)):\;j\neq i\})$
with, say, $\lfloor \delta n/2\rfloor\kmax(\{\dist(R_j(A),H^{j}(A)):\;j\neq i,\,j\leq \delta n\})$
in the events' definitions. Then, literally repeating the above argument
(essentially ignoring all indices $i>\delta n$ and working only with $i\leq \delta n$), 
we would be able to get an upper bound $\Prob(\Event)\leq \frac{C\log_2^C n}{\delta t}$.
This, however, is still not satisfactory when the parameter $\delta$ decays with $n$.
To handle the problem, in addition to event partitions we will introduce auxiliary partitions of the
probability space which will serve as ``weights'' in a more general analog of relation \eqref{eq: aux 09827520560298},
that we will derive in the next subsection.
We will then return to discussing the example.

\subsection{An $(\alpha,\eta)$--structure on a product space}\label{s: alphaeta}

Let $(\Omega,\mathcal F,\Prob)=\prod_{i=1}^n (\Omega_i,\mathcal F_i,\Prob_i)$ be a product probability space.
In this subsection, we define notion of an $(\alpha,\eta)$--structure on $\Omega$,
and prove a basic lemma which will provide a way to estimate the small ball probability 
of $s_{\min}$ for non-centered random matrices.

Let $\Psi$ and $\Lambda$ be two finite index sets.
For every $i\leq n$, let
$$\bigsqcup_{\psi\in\Psi}\Class_{i,\psi}=\Omega$$
be a partition of $\Omega$ into measurable subsets (events), and for every $\psi\in\Psi$
denote by $\sharp\psi$ the minimal integer such that for any
subset $I\subset[n]$ with $|I|\geq \sharp\psi+1$ we have
$$\bigcap\limits_{i\in I} \Class_{i,\psi}=\emptyset$$
(if such a number does not exist, we set $\sharp\psi=n$).
We note that the functions
$$
(\omega_1,\omega_2,\dots,\omega_n)\longrightarrow
\Prob_i\{\widetilde \omega_i\in\Omega_i:\,(\omega_1,\dots,\omega_{i-1},\widetilde\omega_i,\omega_{i+1},\dots,\omega_n)
\in\Class_{i,\psi}\},\quad i\leq n,
$$
are measurable. Hence, there exists a measurable mapping $\eta:[n]\times\Omega\to \Psi$
such that $\eta(i,\cdot)$ is measurable w.r.t $\sigma(\mathcal F_j,\, j\neq i)\times \{\emptyset,\Omega_i\}$, and
\begin{align*}
\eta &\big(i,(\omega_1,\omega_2,\dots,\omega_n)\big)=\\
&\argmax\big\{\Prob_i\{\widetilde \omega_i\in\Omega_i:\,(\omega_1,\dots,\omega_{i-1},\widetilde\omega_i,\omega_{i+1},\dots,\omega_n)
\in\Class_{i,\psi}\},\;\;\psi\in\Psi\big\},
\end{align*}
i.e.\ for all $\psi^*\in\Psi$ we have 
$\Prob_i\{\widetilde \omega_i\in\Omega_i:\,(\omega_1,\dots,\omega_{i-1},\widetilde\omega_i,\omega_{i+1},\dots,\omega_n)
\in\Class_{i,\psi^*}\}\leq
\Prob_i\{\widetilde \omega_i\in\Omega_i:\,(\omega_1,\dots,\omega_{i-1},\widetilde\omega_i,\omega_{i+1},\dots,\omega_n)
\in\Class_{i,\eta(i,\omega)}\}$.
In particular,
$\Prob_i\{\widetilde \omega_i\in\Omega_i:\,(\omega_1,\dots,\omega_{i-1},\widetilde\omega_i,\omega_{i+1},\dots,\omega_n)
\in\Class_{i,\eta(i,(\omega_1,\dots,\omega_n))}\}\geq |\Psi|^{-1}$ for all $(\omega_1,\dots,\omega_n)\in\Omega$.

Further, let $\Event\subset\Omega$ be an event, and for every $i\leq n$, let
$\{\Event_{i,\lambda}\}_{\lambda\in\Lambda}$ be a partition of $\Event$ into measurable sets.
Let $\alpha:[n]\times\Event\to\R_+\cup\{\infty\}$ be defined piecewise on $\Event_{i,\lambda}$ by
\begin{align*}
&\alpha\big(i,\omega\big):=
\frac{1}{\Prob_i\{\widetilde \omega_i\in\Omega_i:\,(\omega_1,\dots,\omega_{i-1},\widetilde\omega_i,\omega_{i+1},\dots,\omega_n)
\in\Event_{i,\lambda}\}},\\
&\hspace{7cm} \omega=(\omega_1,\omega_2,\dots,\omega_n)\in\Event_{i,\lambda},\;\;
i\leq n,\;\;\lambda\in\Lambda
\end{align*}
(it follows from Tonelli's theorem that the function is measurable).

Note that both $\eta(i,(\omega_1,\omega_2,\dots,\omega_n))$ and $\alpha(i,(\omega_1,\omega_2,\dots,\omega_n))$
do not depend on the $i$--th coordinate $\omega_i$.
An event $\Event$, index sets $\Psi,\Lambda$ and collections $\{\Class_{i,\psi}\}$, $\{\Event_{i,\lambda}\}$, together
with the functions $\alpha(\cdot,\cdot)$ and $\eta(\cdot,\cdot)$, satisfying all of the above conditions
(including measurability),
will be called an {\it $(\alpha,\eta)$--structure on $\Omega$}.

\begin{lemma}[A property of $(\alpha,\eta)$--structures]\label{l: alpharho}
Let $\Omega=\prod_{i=1}^n\Omega_i$ be a product probability space, and fix any $(\alpha,\eta)$--structure on $\Omega$.
Then
$$\int\limits_{\Event}\sum_{i=1}^n\frac{\alpha(i,\omega)}{\sharp\eta(i,\omega)}\,d\Prob(\omega)
\leq |\Psi|^2\,|\Lambda|.$$
\end{lemma}
\begin{proof}
First, for any $i\leq n$ and $\lambda\in\Lambda$, applying the definitions of
$\alpha(\cdot,\cdot)$ and $\eta(\cdot,\cdot)$, we get
\begin{align*}
\int\limits_{\Event_{i,\lambda}}&\frac{\alpha(i,\omega)}{\sharp \eta(i,\omega)}\,d\Prob(\omega)\\
&=\int\limits_{\Event_{i,\lambda}}\frac{d\Prob_1(\omega_1)\, d\Prob_2(\omega_2)\dots d\Prob_n(\omega_n)}
{\sharp \eta(i,(\omega_1,\dots,\omega_n))\,\Prob_i\{\widetilde \omega_i:\,(\omega_1,\dots,\omega_{i-1},\widetilde\omega_i,\omega_{i+1},\dots,\omega_n)
\in\Event_{i,\lambda}\}}\\
&\leq |\Psi|
\int\limits_{\Event_{i,\lambda}}\frac{
\Prob_i\{\widetilde \omega_i:\,(\omega_1,\dots,\omega_{i-1},\widetilde\omega_i,\omega_{i+1},\dots,\omega_n)
\in\Class_{i,\eta(i,(\omega_1,\dots,\omega_n))}\}}
{\sharp \eta(i,(\omega_1,\dots,\omega_n))\,\Prob_i\{\widetilde \omega_i:\,(\omega_1,\dots,\omega_{i-1},
\widetilde\omega_i,\omega_{i+1},\dots,\omega_n)
\in\Event_{i,\lambda}\}}\,d\Prob(\omega)\\
&\leq |\Psi|
\sum_{\psi\in\Psi}\,\int\limits_{\Event_{i,\lambda}}\frac{
\Prob_i\{\widetilde \omega_i:\,(\omega_1,\dots,\omega_{i-1},\widetilde\omega_i,\omega_{i+1},\dots,\omega_n)
\in\Class_{i,\psi}\}}
{\sharp\psi\,\Prob_i\{\widetilde \omega_i:\,(\omega_1,\dots,\omega_{i-1},
\widetilde\omega_i,\omega_{i+1},\dots,\omega_n)
\in\Event_{i,\lambda}\}}\,d\Prob(\omega).
\end{align*}

Passing to an iterated integral, we get from the last relation
$$
\int\limits_{\Event_{i,\lambda}}\frac{\alpha(i,\omega)}{\sharp\eta(i,\omega)}\,d\Prob(\omega)
\leq |\Psi|
\sum_{\psi\in\Psi}\frac{\Prob(\Class_{i,\psi})}{\sharp\psi}.
$$
Summing over all $\lambda\in\Lambda$, we obtain
\begin{align*}
\int\limits_{\Event}\frac{\alpha(i,\omega)}{\sharp\eta(i,\omega)}\,d\Prob(\omega)
\leq |\Psi|\,|\Lambda|
\sum_{\psi\in\Psi}\frac{\Prob(\Class_{i,\psi})}{\sharp\psi},\quad\quad i\leq n.
\end{align*}
Further, for every $\psi$ the intersection of any collection of $\sharp\psi+1$ events $\Class_{i,\psi}$ is empty,
according to the definition
of an $(\alpha,\eta)$--structure. Hence,
$$\sum_{i=1}^n \Prob\big(\Class_{i,\psi}\big)\leq \sharp\psi,$$
and we get
\begin{align*}
\int\limits_{\Event}\sum\limits_{i=1}^n\frac{\alpha(i,\omega)}{\sharp\eta(i,\omega)}\,d\Prob(\omega)
\leq |\Psi|\,|\Lambda|
\sum_{\psi\in\Psi}1=|\Psi|^2|\Lambda|.
\end{align*}
\end{proof}

Lemma~\ref{l: alpharho} allows to estimate probability of the event $\Event$ as long as a uniform lower
bound for $\sum\limits_{i=1}^n\frac{\alpha(i,\omega)}{\sharp\eta(i,\omega)}$ is known;
say, if $\sum\limits_{i=1}^n \frac{\alpha(i,\omega)}{\sharp\eta(i,\omega)}\geq \tau$ for some
$\tau>0$ almost everywhere on $\Event$
then necessarily $\Prob(\Event)\leq \tau^{-1}|\Psi|^2|\Lambda|$.
Note that in the case $|\Psi|=1$ (and hence $\Class_{i}=\Omega$, $i\leq n$) the lemma
is reduced to relation \eqref{eq: aux 09827520560298} considered in the previous subsection.
By introducing the classes $(\Class_{i,\psi})$ we add more flexibility to the method.

As an illustration, let us continue with the example at the end of the previous subsection.
We have the matrix $A+M$, where the first $\delta n$ rows of $M$ are zeros,
and the remaining $n-\delta n$ rows are pairwise orthogonal and have very large Euclidean norms.
To provide satisfactory probability estimates for the event $\Event:=\big\{\|(A+M)^{-1}\|_{HS}\in[t\sqrt{n},2t\sqrt{n}]\big\}$,
we could apply the argument from that subsection, by replacing
$\lfloor n/2\rfloor\kmax(\{\dist(R_j(A),H^{j}(A)):\;j\neq i\})$
with, say, $\lfloor \delta n/2\rfloor\kmax(\{\dist(R_j(A),H^{j}(A)):\;j\neq i,\,j\leq \delta n\})$
in the definition of $\Event_{i,\lambda}$ (including appropriate changes in the definition of $\Event_{i,\infty_1}$). However,
application of \eqref{eq: aux 09827520560298} then introduces an excessive ``$1/\delta$''
multiple in the estimate of $\Prob(\Event)$.
This can be easily fixed with help of Lemma~\ref{l: alpharho}:
we define $\Psi=\{0,1\}$, and set $\Class_{i,0}=\emptyset$, $\Class_{i,1}=\Omega$
for all $i\leq \delta n$, and $\Class_{i,1}=\emptyset$, $\Class_{i,0}=\Omega$
for all $i>\delta n$.
It is easy to see that $\sharp\eta(i,\omega)=\delta n$\label{p: classes} for all $\omega\in\Event$
and $i\leq \delta n$, so that the lemma gives
$$
\int\limits_{\Event}\bigg(\sum_{i\leq \delta n}\frac{1}{\delta}\,\alpha(i,\omega)\,d\Prob(\omega)
+\sum_{i> \delta n}\frac{1}{1-\delta}\,\alpha(i,\omega)\,d\Prob(\omega)\bigg)
\leq 4n\,|\Lambda|.
$$
Thus, compared to \eqref{eq: aux 09827520560298}, we added different ``weights'' to indices $i\leq \delta n$
and $i>\delta n$.
It can be checked now that, by repeating the argument from Subsection~\ref{subs:event} (and using the modified
definition of the event partition),
we would be able to get correct (up to the polylog multiples) upper bound for $\Prob(\Event)$.
We omit the details.

Obviously, the above example can be addressed without using Lemma~\ref{l: alpharho},
and the structures $(\Class_{i,\psi})$ may seem excessive.
However, in the setting of arbitrary non-zero shift $M$, it is not clear at all how to locate ``the right'' row indices
to define the events $\Event_{i,\lambda}$, and whether this choice can be non-random.
That is why the actual definition of the classes is much more complicated than in the above example. 

\bigskip

Finally, let us briefly sketch the actual argument used to prove the main statement of the section.
Our decomposition of the event $\{\|A+M\|_{HS}^{-1}\geq t\sqrt{n}\}$ is similar
to the one considered in Subsection~\ref{subs:event}, with two principal differences:
(a) Instead of working with pairs of indices $i,j$ we study triples, allowing to get better
anti-concentration estimates and to remove excessive logarithmic error factors from the estimates
(we recall that in the proof of Theorem~\ref{th:revisiting} we had to consider $4$--tuples
of rows for essentially the same reason).
(b) Instead of working with a fixed collection of indices $P_i\subset[n]\setminus\{i\}$ in the expressions
$\lfloor |P_i|/2\rfloor\kmax(\{\dist(R_j(A+M),H^{j}(A+M)):\;j\in P_i\})$ (say, we had $P_i=[n]\setminus\{i\}$
in Subsection~\ref{subs:event} or $P_i=[\delta n]\setminus\{i\}$ in the modified example above),
we consider specially constructed {\it random} subsets (of pairs of indices).
As we mentioned above, in the general setting of arbirary shift $M$, it is completely unclear
to us if it is possible to get a useful {\it deterministic} classification of the row indices
and use it in the event partitioning. That is why our designated sets of pairs $P_i$
fluctuate depending on a point of the probability space.

To summarize, the actual event partitions will have the form
\begin{align*}
\big\{&\dist(R_i(A+M),H^i(A+M))\in [2^{\lambda_1}t,2^{\lambda_1+1}t),
\;\rhonew_i\neq\emptyset,\mbox{ and}\\
&\lceil|\rhonew_i|/2\rceil\kmax\big(\big\{
\min\big(\dist(R_k(A+M),H^k(A+M)),\dist(R_\ell(A+M),H^\ell(A+M))\big)
,\\
&\{k,\ell\}\in\rhonew_i\big\}\big)\in[2^{\lambda_2}t,2^{\lambda_2+1}t)
\big\},
\end{align*}
where $\rhonew_i$ are certain random sets of unordered pairs of indices $\{k,\ell\}$.
The full definition will be given in Subsection~\ref{ss:actual def}.

\medskip

The second principal component of the argument is the definition of classes $\Class_{i,\psi}$
which is used to appropriately ``re-weigh'' the relation \eqref{eq: aux 09827520560298}.
By checking the simple example of the shift $M$ considered above, one could argue
that a most efficient way to ``re-weigh'' \eqref{eq: aux 09827520560298} (to obtain the strongest
possible estimate for $\Prob(\Event)$) is to assign larger weights to those indices $i$ for which corresponding
distance between $R_i(A+M)$ the the span of other rows tends to be small.
If all distances are typically more or less comparable, this means no re-weighing is necessary (just as in case
of centered matrices). However, if for a large number of indices corresponding distances ``often'' tend to be large,
then for the remaining indices the weights should be made larger. 
In terms of the classes $\Class_{i,\psi}$, we would expect the number $\sharp\eta(i,\omega)$
from Lemma~\ref{l: alpharho} to be smallest possible whenever $i$ and $\omega=(\omega_1,\dots,\omega_n)$ are such that
the distance from $R_i(A(\omega_1,\dots,\omega_{i-1},\widetilde \omega_i,\omega_{i+1},\dots,\omega_n)+M)$
to the span of $R_j(A(\omega_1,\dots,\omega_n)+M)$, $j\neq i$,
is typically not large. The key concept we use to implement this strategy is that of {\it vertex value}.

Given $i\leq n$, the vertex value of $i$ is a random scalar quantity which characterizes
how the distance from $i$-th row to span of other rows is related to other distances at the same point of the
probability space. More specifically, for every $i\leq n$ we define random undirected graph $\Graph_i$ on $[n]$
by taking its edge set to be all the unordered pairs $\{j,k\}$ ($j\neq k\neq i$) satisfying
\begin{align*}
\dist(&R_i(A+M),H^{\{i,j,k\}}(A+M))\\
&\geq \max\big(\dist(R_j(A+M),H^{\{i,j,k\}}(A+M)),\dist(R_k(A+M),H^{\{i,j,k\}}(A+M))\big).
\end{align*}
Then the vertex value of $i$ can be very approximately defined as the square root of the number of edges of $\Graph_i$
(the actual definition is more complicated and is given in the next subsection in an abstract non-random setting).
The key property of the vertex value is revealed in Lemma~\ref{l: value det}:
it turns out that for all $N\in\N$, the number of indices $i\leq n$ with the value $\vl_i\leq N$
does not exceed $16N$. At the qualitative level, this should match the intuition: the fact that the edge set of
the graph $\Graph_i$ is small (hence, the value $\vl_i$ is small) means that the distance
to $i$-th row from corresponding linear spans is typically smaller than for other matrix rows.
Thus, the total number of such indices $i$ cannot be large.
The vertex value allows to define the classes $\Class_{i,\psi}$:
$$
\Class_{i,\psi}:=\big\{\omega\in\Omega:\,\vl_i(\omega)\in(2^{-\psi-1}n,2^{-\psi}n]\big\},\quad i\leq n,\;\; \psi\geq 0
$$
(see Subsection~\ref{ss:actual def} for full definition).
Then the above property of the vertex value implies $\sharp \psi\leq 16\cdot 2^{-\psi} n$, $\psi\geq 0$
(see Lemma~\ref{l: value random}), thus, the ``weight'' we assign to 
$\alpha(i,\omega)$ when applying Lemma~\ref{l: alpharho} is proportional
to the inverse of the vertex value of $i$ at a given point of the probability space.
Note that this perfectly agrees with our intention: the rows with typically smaller distances to respective linear
spans are ``rewarded'' with relatively larger weights, allowing more efficient estimate of $\Prob(\Event)$.

Returning to our example with the matrix $M$ having first $\delta n$ rows equal zero and last $n-\delta n$ rows
pairwise orthogonal and with large Euclidean norms, we would get that the graphs $\Graph_i$ for $i\leq \delta n$
typically do not have edges emanating from $\{\delta n+1,\dots, n\}$, so the vertex value of $i$ is typically of order $O(\delta n)$.
This, in turn, would imply that $\sharp \eta(i,\omega)$ is of order $O(\delta n)$ for each $i\leq \delta n$
on most of the probability space, which essentially agrees with our explicit construction of classes for this example
on page~\pageref{p: classes}.

\medskip

As we have mentioned, the actual definition of the vertex value is more complicated
than just taking square root of the cardinality of the edge set of $\Graph_i$.
The definition should be given in such a way as to guarantee validity of Lemma~\ref{l: value det},
but also to link the properties of the graph $\Graph_i$ with the random subsets $\rhonew_i$
employed in the event partitions.
In the next subsection, we give a precise definition of the vertex value in the deterministic context of arbitrary
undirected graphs, and verify some of its basic properties.
The random counterpart of those properties is provided in Subsection~\ref{ss:actual def}.

\subsection{Auxiliary graph constructions}

Defining our $(\alpha,\eta)$--structure for matrices requires some preparatory work.
In this subsection, we will consider some properties of deterministic graphs related to counting the number
of incident edges/vertices. The main objects defined in this part of the paper --- the {\it value}
of a vertex $i$, and the edge collection $\rhonew(i)$ --- will be used to construct the classes $\Class_{i,\psi}$
and events $\Event_{i,\lambda}$.

\bigskip

Let $\Graph=([n],E)$ be any undirected simple graph on $[n]$ with an edge set $E$.
With the graph $\Graph$ 
we associate a sequence of subsets
$\emptyset=S_0(\Graph)\subset S_1(\Graph)\subset S_2(\Graph)\dots\subset [n]$
and another sequence of edge collections
$E=E_0(\Graph)\supset E_1(\Graph)\supset E_2(\Graph)\dots$, both
constructed inductively as follows.
\begin{equation}\label{eq: SE definitions}
\parbox{14cm}
{At the first step, we let $S_1=S_1(\Graph)$
be a smallest possible subset of vertices of $\Graph$
such that $E_1(\Graph):=\{e\in E:\;\mbox{$e$ is not incident to any vertex in $S_1$}\}$
has cardinality
less than or equal to $|E|/2$. At $k$--th step, we define $S_k=S_k(\Graph)$ as a subset of $[n]$
of the smallest
possible cardinality such that
$S_k\supset S_{k-1}$ and the collection
of edges
$E_k(\Graph):=\{e\in E:\;\mbox{$e$ is not incident to any vertex in $S_k$}\}$
has cardinality
less than or equal to $|E_{k-1}(\Graph)|/2$.}
\end{equation}

Later, it will be important for us to assume that the choice of the sequences
$(S_k(\Graph))_{k=0}^\infty$ and $(E_k(\Graph))_{k=0}^\infty$ is unique for every graph $\Graph$.
The uniqueness can be achieved by fixing any total order on the set of subset of $[n]$
and, at every step above, choosing ``the greatest'' admissible subset $S_k(\Graph)$
with respect to that order.
The cardinalities of the sets $|S_k(\Graph)|$ can be viewed as a characterization of
the number of incidences between the edges of $\Graph$: if all edges of $\Graph$ emanate
from a set of few vertices then $|S_k(\Graph)|$ are small, otherwise --- large.
Considering the whole sequence $(|S_k(\Graph)|)_{k=1}^\infty$ allows a more delicate analysis:
for example, if $|S_k(\Graph)|$ are small for $k\leq k_0$ and large for $k>k_0$, this means
that the graph $\Graph$ contains a small set of high-degree vertices which are incident to roughly
$(1-2^{-k_0})\cdot 100\%$ of the graph edges, whereas the rest of the edges is scattered on low-degree vertices.
The magnitudes of $|S_k(\Graph)|$, as well as the edge collections $(E_k(\Graph))_{k=0}^\infty$,
are crucial for our construction.

The next simple observation will be useful:
\begin{lemma}\label{l: aux 0y}
Let $\Graph$ be a simple graph on $[n]$ 
and let the sequence $(E_k(\Graph))_{k=0}^\infty$ be defined as above. Then for any $k\geq 1$ and $0<\ell\leq k$ we have
$$2^{\ell}|E_{k}(\Graph)|\leq |E_{k-\ell}(\Graph)|\leq 2^{\ell}|E_{k}(\Graph)|+2^{\ell+1}n.$$
\end{lemma}
\begin{proof}
First, assume that for some $m\geq 1$ we have $|E_{m-1}(\Graph)|> 2|E_{m}(\Graph)|+2n$.
Take $S:=S_{m}(\Graph)\setminus S_{m-1}(\Graph)$, and denote by $S'$
any subset of $S$ with $|S\setminus S'|=1$.
Then the collection of edges
$$E':=\big\{e\in E_0(\Graph):\;\mbox{$e$ is not incident to any vertex in $S'\cup S_{m-1}(\Graph)$}\big\}$$
has cardinality at most $|E_{m}(\Graph)|+n< \frac{1}{2}|E_{m-1}(\Graph)|$, since any given vertex is incident to at most $n$ edges.
This contradicts the choice of $S_m(\Graph)$; thus, for any $m\geq 1$ we have
$$|E_{m-1}(\Graph)|\leq 2|E_{m}(\Graph)|+2n.$$

It remains to make an inductive step:
if for some $k\geq 1$ and $0<\ell\leq k$ we have
$|E_{k-(\ell-1)}(\Graph)|\leq 2^{\ell-1}|E_{k}(\Graph)|+\sum_{r=1}^{\ell-1}2^{r}n$
then, in view of the above assertion,
\begin{align*}
|E_{k-\ell}(\Graph)|\leq 2|E_{k-\ell+1}(\Graph)|+2n\leq
2^{\ell}|E_{k}(\Graph)|+2\sum_{r=1}^{\ell-1}2^{r}n+2n
= 2^{\ell}|E_{k}(\Graph)|+\sum_{r=1}^{\ell}2^{r}n.
\end{align*}
The result follows.
\end{proof}

\medskip

Now, assume that we have two simple undirected graphs $\Graph=([n],E)$ and $\widetilde \Graph=([n],\widetilde E)$
on $[n]$. 
Given a positive integer parameter $L$, define {\it the value} of the 
graph $\Graph$ as
$$\vl_{L}(\Graph):=\min\big(\max\big(2^{-L/2}\sqrt{|E|},|S_L(\Graph)|\big),\sqrt{|E|}\big).$$
Note that, without the $|S_L(\Graph)|$ on the right hand side, the value of the graph would be a multiple
of the square root of the number of the graph edges, just as we described the quantity in our overview of the proof in
the previous subsection. However, the term $|S_L(\Graph)|$ is crucial in proving of the main property of the graph value
--- Lemma~\ref{l: value det}.

Second, for the same parameter $L$ we define the set
$\rhonew_{L}(\widetilde\Graph):=E_{\widetilde k_0-1}(\widetilde\Graph)$,
where $1\leq \widetilde k_0\leq 4L$ is the smallest index such that
$$\big|S_{\widetilde k_0}(\widetilde\Graph)\setminus S_{\widetilde k_0-1}(\widetilde\Graph)\big|
=\max_{1\leq k\leq 4L}\big|S_{k}(\widetilde\Graph)\setminus S_{k-1}(\widetilde\Graph)\big|.$$
The precise definition of the graphs $\Graph$ and $\widetilde\Graph$ in context of matrices will be given later;
for now we remark that in the random setting the edge set of the graph $\widetilde\Graph$
will contain almost all edges of $\Graph$, and that
$\rhonew_{L}(\widetilde\Graph)$
are the sets of pairs of indices (vertices) which will be used in the definition of the event partitions
for our Lemma~\ref{l: alpharho}.
Note that the index $\widetilde k_0$ from the definition of $\rhonew_{L}(\widetilde\Graph)$ corresponds
to the largest gap between consecutive elements of the sequence $(|\widetilde S_k|)_{k=0}^{4L}$.


\begin{lemma}\label{l: two graphs det}
Let $\Graph=([n],E)$ and $\widetilde \Graph=([n],\widetilde E)$ be two simple graphs on $[n]$,
let $L\in\N$ be any natural number, and assume that $|E\setminus\widetilde E|\leq 16^{-L}n^2$.
Then at least one of the following two assertions is true:
\begin{itemize}

\item For any subset $I$ of $[n]$ such that at least half of edges in $\rhonew_{L}(\widetilde\Graph)$
are incident to some vertices in $I$, we have
$|I|\geq \frac{1}{4L^2}\vl_{L}(\Graph)$;

\item $\vl_{L}(\Graph)\leq 4\cdot 2^{-L/2}n$.

\end{itemize}
\end{lemma}
\begin{proof}
For brevity, we will denote sets $S_{k}(\Graph)$, $E_k(\Graph)$ by $S_k$, $E_k$;
and the sets $S_{k}(\widetilde\Graph)$, $E_k(\widetilde \Graph)$ by $\widetilde S_k$, $\widetilde E_k$.
Let $k_0\leq L$ be an index corresponding to the difference set $S_{k_0}\setminus S_{k_0-1}$
of cardinality $\max_{1\leq k\leq L}\big|S_k\setminus S_{k-1}\big|$.
Consider two cases.

\medskip

First, assume that $|\widetilde S_{4L}|\geq |S_{k_0}\setminus S_{k_0-1}|$.
Let $1\leq \widetilde k_0\leq 4L$ be the smallest index such that
$$|\widetilde S_{\widetilde k_0}\setminus\widetilde S_{\widetilde k_0-1}|
=\max\limits_{1\leq k\leq 4L}|\widetilde S_{k}\setminus\widetilde S_{k-1}|.$$
Observe that the definition of the set $\rhonew_{L}(\widetilde\Graph)=\widetilde E_{\widetilde k_0-1}$
implies that for any set of vertices $I$ such that at least half of edges in $\rhonew_{L}(\widetilde\Graph)$
are incident to some vertices in $I$, necessarily $|I|\geq |\widetilde S_{\widetilde k_0}\setminus \widetilde S_{\widetilde k_0-1}|$.
Together with the above estimate and the assumption on $\widetilde S_{4L}$, this implies
$$|I|\geq \frac{1}{4L}|\widetilde S_{4L}|\geq \frac{1}{4L^2}|S_{L}|.$$
Hence, as long as $\vl_{L}(\Graph)\leq |S_L|$, we obtain the first assertion of the lemma.
On the other hand, if $\vl_{L}(\Graph)>|S_L|$ then necessarily $\vl_{L}(\Graph)=2^{-L/2}\sqrt{|E|}$,
whence $\vl_{L}(\Graph)\leq 2^{-L/2}n$, and the second assertion of the lemma is satisfied.

\medskip

Otherwise, assume that $|\widetilde S_{4L}|< |S_{k_0}\setminus S_{k_0-1}|$.
Since subset $\widetilde E_{4L}\subset \widetilde E$ has cardinality at most $16^{-L}n^2$,
and $|E\setminus\widetilde E|\leq 16^{-L}n^2$, we obtain that
$$E':=\big\{e\in E:\;\mbox{$e$ is not incident to any vertex in $\widetilde S_{4L}\cup S_{k_0-1}$}\big\}$$
has cardinality at most $2\cdot 16^{-L}n^2$.
Taking into account the construction procedure for $S_{k_0}$, we conclude that, under our assumption, $|E'|>\frac{1}{2}|E_{k_0-1}|$,
whence $2\cdot 16^{-L}n^2>\frac{1}{2}|E_{k_0-1}|$. Obviously, $E_{k_0-1}$ is non-empty under our current assumption
(otherwise, we would have $|\widetilde S_{4L}|= |S_{k_0}\setminus S_{k_0-1}|=0$),
so the last inequality implies that $L\leq \frac{1}{2}\log_2(2 n)$.
Further, by Lemma~\ref{l: aux 0y}, we have $2^{k_0-1}|E_{k_0-1}|+ 2^{k_0}n\geq |E|$.
Hence, we obtain from the previous relation that
$4\cdot 16^{-L}n^2\cdot 2^L+2^{L+1}n\geq |E|$.
Applying the upper bound for $L$, we get $|E|\leq 4\cdot 8^{-L}n^2+2\cdot 2^{2L}2^{-L}n
\leq 8\cdot 2^{-L}n^2$, whence $\vl_{L}(\Graph)\leq 4\cdot 2^{-L/2}n$.
The result follows.
\end{proof}


\begin{lemma}\label{l: value det}
Let $B$ be an $n\times n$ non-random matrix, and fix any positive integer $L$. 
For any $i\leq n$, let $\Graph_{B,i}$ be the simple graph on $[n]$ with the edge set consisting of all unordered pairs
$\{j,k\}$ (so that $j,k\neq i$) satisfying
$$\dist(R_i(B),H^{\{i,j,k\}}(B))
\geq \max\big(\dist(R_j(B),H^{\{i,j,k\}}(B)),\dist(R_k(B),H^{\{i,j,k\}}(B))\big).
$$
Then for all $N\in\N$ we have
$$\big|\big\{i\leq n:\,\vl(i)\leq N\big\}\big|\leq 16N,$$
where $\vl(i)=\vl_{L}(\Graph_{B,i})$ is defined above.
\end{lemma}
\begin{proof}
For each $i\leq n$, denote by $(S_k(i))_{k\geq 0}$ the subsets constructed according to \eqref{eq: SE definitions}
for the graph $\Graph_{B,i}$,
and let $E(i)$ be the edge set of $\Graph_{B,i}$.
Assume that for some $N\geq 1$ the set
$$\big\{i\leq n:\,\vl(i)\leq N\big\}$$
has cardinality greater than $16N$.
For any $r\in\N$, denote by $J^r$ the subset $\big\{i\leq n:\,\max\big(2^{-r/2}\sqrt{|E(i)|},|S_r(i)|\big)\leq N\big\}$.
Then, in view of the definition of the value and the above assumption, either
$|J^0|> 8N$ or $|J^L|> 8N$. Pick $r\in\{0,L\}$ for which the inequality holds true.

Let $W$ be the set of all ordered triples $(i,j,k)$, with pairwise distinct components, where $i,j,k\in J^r$.
Now, define a random triple $t=(t_1,t_2,t_3)$ uniformly distributed in $W$.
Note that for any $i\in J^r$,
we have
$$\Prob\big\{t_2\in S_r(i)\mbox{ or }t_3\in S_r(i)\,|\,t_1=i\big\}\leq \frac{2|S_r(i)|}{|J^r|-1}\leq \frac{2N}{|J^r|-1}.$$
Further, note that for any $i\in J^r$ the size of the set of edges of graph $\Graph_{B,i}$ which are not incident to
any vertex in $S_r(i)$, is at most $2^{-r}|E(i)|$, by the construction of the sequence $(S_k(i))_{k=0}^\infty$.
Hence, conditioned on the event $\{t_1=i\mbox{ and }t_2,t_3\notin S_r(i)\}$, the probability
that $(t_2,t_3)$ belongs to $E(i)$, is at most $\frac{2^{-r+1}|E(i)|}{(|J^r|-1-|S_r(i)|)(|J^r|-2-|S_r(i)|)}$.
Combined with the last inequality and the assumption, this gives
\begin{align*}
\Prob\big\{&(t_2,t_3)\in E(i)\,|\,t_1=i\big\}\\
&\leq
\Prob\big\{t_2\in S_r(i)\mbox{ or }t_3\in S_r(i)\,|\,t_1=i\big\}
+\Prob\big\{(t_2,t_3)\in E(i)\mbox{ and } t_2,t_3\notin S_r(i)\,|\,t_1=i\big\}\\
&\leq \frac{2N}{|J^r|-1}+\frac{2N^2}{(|J^r|-1-N)(|J^r|-2-N)}\\
&< \frac{1}{3}.
\end{align*}
Thus,
$$\Prob\big\{(t_2,t_3)\in E(t_1)\big\}<\frac{1}{3}.$$
Clearly, analogous probability estimates are true for pairs $(t_1,t_3)$ and $(t_1,t_2)$:
\begin{align*}
\Prob\big\{(t_1,t_3)\in E(t_2)\big\}&<\frac{1}{3};\\
\Prob\big\{(t_1,t_2)\in E(t_3)\big\}&<\frac{1}{3}.
\end{align*}
On the other hand, the following is true {\it deterministically}:
\begin{align*}
(t_2,t_3)\in E(t_1)\;\mbox{ or }\;(t_1,t_3)\in E(t_2)\;\mbox{ or }\;(t_1,t_2)\in E(t_3).
\end{align*}
Hence, the above probabilities must sum up to at least one. 
The contradiction shows that the assumption on the cardinality of $J^r$ was wrong.
\end{proof}

\subsection{An $(\alpha,\eta)$--structure for non-centered random matrices}\label{ss:actual def}

Fix parameters $K_1\geq 1$, $K_2:=2000$, a large $n\in\N$ (we assume that $n\geq n_0(K_1)$
for some $n_0(K_1)$) and let $A$ be an $n\times n$ random matrix with
independent rows, so that each row satisfies condition \eqref{C2}
with $K_1,K_2$.
We will view the underlying probability space $(\Omega,\Prob)$ as a product $(\Omega,\Prob)=\prod_{i=1}^n(\Omega_i,\Prob_i)$
where $(\Omega_i,\Prob_i)$ is the domain of the $i$-th row of $A$ ($i\leq n$).
Fix a non-random matrix $M$.
For $t\in(0,1]$ and $u\in\{0,1,\dots,\lfloor \log_2 n\rfloor\}$, define event
\begin{equation}\label{eq: aux pang}
\Event(t,u):=\big\{\exists I\subset[n]\mbox{ with }|I|\geq 2^u\mbox{ s.t.\ }\dist(R_i(A+M),H^i(A+M))\leq t\mbox{ for all }i\in I\big\}.
\end{equation}
All the parameters $n,K_1,K_2,t,u,M$ are fixed throughout the subsection,
and our goal is to estimate the probability of the event $\Event(t,u)$ using the concept of the $(\alpha,\eta)$--structure.

\bigskip

We will start by defining auxiliary random graphs $\Graph_{i}$ and $\widetilde \Graph_i$ corresponding to the
deterministic graphs considered in the previous subsection.
Set
\begin{equation}\label{eq: definitions of parameters}
\varepsilon:=\frac{1}{24},\quad\quad L_u:=8(\lfloor \log_2 n\rfloor+1-u)+2\lceil\log_2(1+K_1)\rceil,\quad\quad \offset_u:=2^{L_u/384}.
\end{equation}

Now, for every $i\leq n$ define the graph $\Graph_i$ on $[n]$ by populating its edge set
with all unordered pairs $\{j,k\}$ ($j\neq k\neq i$) satisfying
\begin{align*}
\dist(&R_i(A+M),H^{\{i,j,k\}}(A+M))\\
&\geq \max\big(\dist(R_j(A+M),H^{\{i,j,k\}}(A+M)),\dist(R_k(A+M),H^{\{i,j,k\}}(A+M))\big).
\end{align*}
Further, the edge set of the random graph $\widetilde \Graph_i$ is defined by taking all couples $\{j,k\}$ ($j\neq k\neq i$)
such that
\begin{align*}
\dist(&R_i(M),H^{\{i,j,k\}}(A+M))+\offset_u\\
&\geq \max\big(\dist(R_j(A+M),H^{\{i,j,k\}}(A+M)),
\dist(R_k(A+M),H^{\{i,j,k\}}(A+M))\big).
\end{align*}
We will sometimes write $\Graph_i(\omega)$ or $\widetilde \Graph_i(\omega)$
for realizations of the respective graphs at a point $\omega$ of the probability space.
The condition for edges in graph $\widetilde\Graph_i$ is in a sense weaker than
for graph $\Graph_i$ because of the presence of the positive parameter $\offset_u$;
in this regard it is reasonable to expect that the edge set of $\widetilde\Graph_i$
will contain almost entire edge set of $\Graph_i$ with high probability
(see Lemma~\ref{l: two graphs prob} below).
The crucial difference between the two graphs
is that the $i$--th row of the matrix $A$ does not participate in the definition of $\widetilde \Graph_i$.
The graph $\widetilde \Graph_i$ (more precisely, the set $\rhonew_i$; see below)
will be used in the definition of the event partitions $(\Event_{i,\lambda})$;
the fact that $\widetilde \Graph_i$ is independent from $R_i(A)$ will be crucial in estimating the function $\alpha(i,\cdot)$.
On the other hand, the vertex value defined below with respect to the graph $\Graph_i$,
will in turn be used to define the classes $\Class_{i,\psi}$.
Thus, $\Graph_i$ and $\widetilde \Graph_i$ are both employed in the definition of the $(\alpha,\eta)$--structure
for the matrix $A+M$. The key aspect of the definition --- that $\Event_{i,\lambda}$ and
$\Class_{i,\psi}$ are ``synchronized'' in the sense that $\sharp \eta(i,\cdot)$ provides
an efficient re-weighing of $\alpha(i,\cdot)$ --- will follow from the fact that with a large probability
the edge set of $\widetilde \Graph_i$ contains almost all of $\Graph_i$.

For every $i\leq n$ we define the {\it vertex value} $\vl_i$ the same way as
in the previous subsection (up to small changes of notation):
$$\vl_{i}:=\min\big(\max\big(2^{-L_u/2}\sqrt{|E(\Graph_i)|},|S_{L_u}(\Graph_i)|\big),\sqrt{|E(\Graph_i)|}\big),$$
where $E(\Graph_i)$ is the edge set of $\Graph_i$, and the sequence of vertex subsets $(S_{k}(\Graph_i))_{k=0}^\infty$
is defined according to \eqref{eq: SE definitions}.
Thus, $\vl_i$ is a random variable taking values in $[0,n]$.
Now, we define the indexing set $\Psi:=\{0,1,2,\dots,L_u\}$ and classes $\Class_{i,\psi}$ as
$$\Class_{i,\psi}:=\big\{\omega\in\Omega:\,\vl_i(\omega)\in(2^{-\psi-1}n,2^{-\psi}n]\big\},\quad i\leq n,\;\; \psi\leq L_u-1,$$
and
$$\Class_{i,L_u}:=\big\{\omega\in\Omega:\,\vl_i(\omega)\leq 2^{-L_u}n\big\},\quad i\leq n.$$
The crucial aspect of the definition --- an estimate for $\sharp \psi$ --- immediately follows from Lemma~\ref{l: value det}.
We restate the lemma in the random setting for convenience:
\begin{lemma}\label{l: value random}
With $\Psi$ and $\Class_{i,\psi}$ as above, we have
$$\sharp \psi\leq 16\cdot \lceil 2^{-\psi}n\rceil$$
for all $\psi\in\Psi$, where $\sharp \psi$ is defined as at the beginning of Subsection~\ref{s: alphaeta}.
\end{lemma}

Further, for every $i\leq n$ define a random subset $\rhonew_i$ of unordered pairs exactly as in the last subsection:
$\rhonew_i:=E_{\widetilde k_0-1}(\widetilde\Graph_i)$,
where $1\leq \widetilde k_0\leq 4L_u$ is the smallest index such that
$$\big|S_{\widetilde k_0}(\widetilde\Graph_i)\setminus S_{\widetilde k_0-1}(\widetilde\Graph_i)\big|
=\max_{1\leq k\leq 4L_u}\big|S_{k}(\widetilde\Graph_i)\setminus S_{k-1}(\widetilde\Graph_i)\big|.$$
The set $\rhonew_i$ will be later used to define the events $\Event_{i,\lambda}$.
For now, we need to establish a relation between $\rhonew_i$ and properties of the classes $\Class_{i,\psi}$.
We start with a preparatory lemma:
\begin{lemma}\label{l: two graphs prob}
Let $K_1,K_2,A,M$ be as above.
Then, denoting by $E(i)$ and $\widetilde E(i)$ the respective edge sets of $\Graph_i$ and $\widetilde\Graph_i$, we have
$$\Prob\big\{|E(i)\setminus \widetilde E(i)|> 16^{-L_u}n^2\,|\,R_k(A),\;k\neq i\big\}\leq K_1\,2^{-L_u},$$
where expression on the left hand side of the inequality is conditional probability given $R_k(A)$ ($k\neq i$).
\end{lemma}
\begin{proof}
We condition on any realization of rows of $A$ except the $i$--th, so that $\widetilde\Graph_i$ is fixed.
For any unordered couple $\{j,k\}$ {\it not contained} in $\widetilde E(i)$ we necessarily have
\begin{align*}
\dist(&R_i(M),H^{\{i,j,k\}}(A+M))+\offset_u\\
&< \max\big(\dist(R_j(A+M),H^{\{i,j,k\}}(A+M)),
\dist(R_k(A+M),H^{\{i,j,k\}}(A+M))\big).
\end{align*}
Hence, $(j,k)$ belongs to $E(i)\setminus \widetilde E(i)$ only if $\dist(R_i(A),H^{\{i,j,k\}}(A+M))\geq\offset_u$.
In view of \eqref{C2}, the last event can happen with probability at most
\begin{align*}
\int\limits_{x\in\R^3:\,\|x\|_2\geq\offset_u}K_1/\max(1,\|x\|_2^{K_2})\,dx
&=\int\limits_{y=\offset_u}^\infty K_1\cdot 4\pi y^2/y^{K_2}\,dy\\
&=\frac{4\pi K_1}{K_2-3}\offset_u^{3-K_2}\\
&\leq K_1\,32^{-L_u},
\end{align*}
where the last inequality follows from the definition of $L_u,\offset_u$ \eqref{eq: definitions of parameters}.
Thus, the conditional expectation of $|E(i)\setminus \widetilde E(i)|$ can be estimated as
$$\Exp|E(i)\setminus \widetilde E(i)|\leq K_1\,32^{-L_u} n^2.$$
It remains to apply Markov's inequality.
\end{proof}

Now, we can prove a random analog of Lemma~\ref{l: two graphs det}:
\begin{lemma}\label{l: rhonew vs eta}
Let $L_u$, the classes $\Class_{i,\psi}$ and random sets $\rhonew_i$ be defined as above, and let $\eta:[n]\times\Omega\to\Psi$
be the function from subsection~\ref{s: alphaeta}.
Then for every $\omega\in\Omega$ at least one of the following two assertions is true:
\begin{itemize}

\item For any subset $I$ of $[n]$ such that at least half of unordered pairs in $\rhonew_i(\omega)$
are incident to some elements of $I$, we have
$|I|\geq \frac{\sharp\eta(i,\omega)}{256L_u^2}$;

\item $\sharp\eta(i,\omega)\leq \max(256\cdot 2^{-L_u/2}n,16)$.

\end{itemize}
\end{lemma}
\begin{proof}
Fix any $\omega=(\omega_1,\omega_2,\dots,\omega_n)\in\Omega$ and $i\leq n$,
and let $E$ and $\widetilde E$ be the edge sets of graphs $\Graph_i$ and $\widetilde\Graph_i$.
Denote
\begin{align*}
\bar\Event_i:=\big\{&\bar\omega_i\in\Omega_i:\\
&|E(\omega_1,\dots,\omega_{i-1},
\bar\omega_i,\omega_{i+1},\dots,\omega_n)\setminus \widetilde E
(\omega_1,\dots,\omega_{i-1},
\bar\omega_i,\omega_{i+1},\dots,\omega_n)|\leq 16^{-L_u}n^2\big\}.
\end{align*}
By Lemma~\ref{l: two graphs prob} and the definition of $L_u$ \eqref{eq: definitions of parameters},
we have $\Prob_i(\bar\Event_i)\geq 1-K_1\,2^{-L_u}>1-\frac{1}{L_u+1}$,
whence $\bar\Event_i$ must have a non-empty intersection with the ``section'' of the class $\Class_{i,\eta(i,\omega)}$:
$$T:=\bar\Event_i\,\cap \,\big\{\omega_i'\in\Omega_i:\,(\omega_1,\dots,\omega_{i-1},
\omega_i',\omega_{i+1},\dots,\omega_n)\in\Class_{i,\eta(i,\omega)}\big\}\neq \emptyset.$$
Fix any $\bar\omega_i\in T$ and set $\bar\omega:=(\omega_1,\dots,\omega_{i-1},
\bar\omega_i,\omega_{i+1},\dots,\omega_n)$.
In view of Lemma~\ref{l: two graphs det}, 
we have either
{\it a)} For any subset $I$ of $[n]$ such that at least half of edges of $\rhonew_i(\bar\omega)$
are incident to some vertices in $I$, necessarily $|I|\geq \frac{1}{4L_u^2}\vl_i(\bar\omega)$, or
{\it b)} $\vl_i(\bar\omega)\leq 4\cdot 2^{-L_u/2}n$.
If {\it (b)} holds for $\bar\omega$ then, in view of the definition of classes $\Class_{i,\psi}$,
we immediately get $\eta(i,\bar\omega)\geq L_u/2-3$,
whence, by Lemma~\ref{l: value random}, $\sharp\eta(i,\bar\omega)\leq \max(256\cdot 2^{-L_u/2}n,16)$.
If {\it (a)} holds and {\it (b)} does not hold then for any subset $I\subset[n]$ satisfying the aforementioned conditions,
$$|I|\geq \max\bigg(1,\frac{1}{4L_u^2}\vl_i(\omega_1,\dots,\omega_{i-1},
\bar\omega_i,\omega_{i+1},\dots,\omega_n)\bigg)
\geq \max\bigg(1,\frac{2^{-\eta(i,\bar\omega)}n}{8L_u^2}\bigg)\geq \frac{\sharp\eta(i,\bar\omega)}{256L_u^2}.$$
Finally, note that $\eta(i,\bar\omega)=\eta(i,\omega)$ and $\rhonew_i(\bar\omega)=\rhonew_i(\omega)$.
The result follows.
\end{proof}

\bigskip

\noindent {\it Definition of the index set $\Lambda$ and events $\Event_{i,\lambda}$.}

We shall define $\Lambda$ as a finite subset of $(\Z\cup\{\pm\infty\})^2$:
\begin{align*}
\Lambda:=\big(\{-\infty\}\cup\{+\infty\}\cup\{-L_u, -L_u+1,\dots,L_u\}\big)^2.
\end{align*}

\medskip

For any unordered pair $\{j,k\}$ denote
\begin{equation*}\label{eq: mindist def}
\mindist(j,k):=\min\big(\dist(R_j(A+M),H^j(A+M)),\dist(R_k(A+M),H^k(A+M))\big).
\end{equation*}
Further, given a non-empty multiset $T$ of real numbers and any $1\leq k\leq |T|$, denote by $k\kmax(T)$
the $k$-th largest element of $T$ (counting multiplicities).
For example, if $T=\{1,1,2,2,4\}$ then $3\kmax(T)=2$, and $4\kmax(T)=1$.

\medskip

Now, for every $i\leq n$ and every pair $\lambda=(\lambda_1,\lambda_2)\in\Lambda\cap\Z^2$ we set
\begin{align*}
\Event_{i,\lambda}:=\Event(t,u)\cap\big\{&\dist(R_i(A+M),H^i(A+M))\in [2^{\lambda_1}t,2^{\lambda_1+1}t),
\;\rhonew_i\neq\emptyset,\\
&\mbox{and }\lceil|\rhonew_i|/2\rceil\kmax(\{\mindist(e),\;e\in\rhonew_i\})\in[2^{\lambda_2}t,2^{\lambda_2+1}t)
\big\}.
\end{align*}

Further, for elements of $\Lambda$ of the form $(\lambda_1,+\infty)$, $\lambda_1\neq \pm\infty$, we set
\begin{align*}
\Event_{i,(\lambda_1,+\infty)}:=\Event(t,u)\cap\big\{&\dist(R_i(A+M),H^i(A+M))\in [2^{\lambda_1}t,2^{\lambda_1+1}t),
\;\rhonew_i\neq\emptyset,\\
&\mbox{and }\lceil|\rhonew_i|/2\rceil\kmax(\{\mindist(e),\;e\in\rhonew_i\})\geq 2^{L_u+1}t)
\big\},
\end{align*}
and for those of the form $(\lambda_1,-\infty)$, $\lambda_1\neq\pm \infty$, we set
\begin{align*}
\Event_{i,(\lambda_1,-\infty)}:=\Event(t,u)\cap\big\{&\dist(R_i(A+M),H^i(A+M))\in [2^{\lambda_1}t,2^{\lambda_1+1}t)\mbox{ and }
\big(\rhonew_i=\emptyset
\;\mbox{ or }\;\\
&\lceil|\rhonew_i|/2\rceil\kmax(\{\mindist(e),\;e\in\rhonew_i\})< 2^{-L_u}t)\big)
\big\}.
\end{align*}
It is easy to see that, for every $i\leq n$ and every integer $\lambda_1$ in $[-L_u,L_u]$, the
events $\{\Event_{i,(\lambda_1,\lambda_2)}\}_{\lambda_2}$ are pairwise disjoint and their union over all admissible $\lambda_2$
is the event
$$\Event(t,u)\cap\big\{\dist(R_i(A+M),H^i(A+M))\in [2^{\lambda_1}t,2^{\lambda_1+1}t)\big\}.$$

We shall define events $\Event_{i,(-\infty,\lambda_2)}$ and $\Event_{i,(+\infty,\lambda_2)}$
in a very similar way, just replacing the condition $\dist(R_i(A+M),H^i(A+M))\in [2^{\lambda_1}t,2^{\lambda_1+1}t)$
in the above definitions with
$$\dist(R_i(A+M),H^i(A+M))<2^{-L_u}t$$
and
$$\dist(R_i(A+M),H^i(A+M))\geq 2^{L_u+1}t,$$
respectively.

\smallskip

Clearly, this way we obtain a partition $\{\Event_{i,\lambda}\}_{\lambda\in\Lambda}$ of $\Event(t,u)$ for every
$i\leq n$, and thus an $(\alpha,\eta)$--structure on $\Omega$ is constructed.

\bigskip

The next (and final) step is to show that, with the structure defined above, we can bound the 
sum $\sum_{i=1}^n\frac{\alpha(i,\omega)}{\sharp\eta(i,\omega)}$ from below
so that an application of Lemma~\ref{l: alpharho} would give a satisfactory upper bound on the probability
of $\Event(t,u)$ (in fact, a somewhat different event will be considered).
Thus, we will obtain deviation estimates for the Hilbert--Schmidt norm of $(A+M)^{-1}$.

\subsection{Proof of the main theorem}

\begin{lemma}\label{l: aux 209t}
Assume that a random vector $X$ in $\R^n$ satisfies condition \eqref{C2} with parameters $K_1>0$ and $K_2\geq 3$.
Then for any $(n-1)$--dimensional linear subspace $F$ of $\R^n$ and any fixed vector $Z\in\R^n$ we have
$$\Prob\big\{\dist(X+Z,F)\leq \tau\big\}\leq C_{\ref{l: aux 209t}}\,K_1\tau,\quad\quad \tau>0,$$
where $C_{\ref{l: aux 209t}}>0$ is a universal constant.
\end{lemma}
\begin{proof}
In view of \eqref{C2}, we have for some fixed numbers $z_1,z_2,z_3\in\R$
\begin{align*}
\Prob\big\{\dist(X+Z,F)\leq \tau\big\}
&\leq \int\limits_{\R^2}\int\limits_{-\tau}^{\tau} \frac{K_1\,d y_1}{\max(1,\|(y_1+z_1,y_2+z_2,y_3+z_3)\|_2^{K_2})}\,dy_2\, dy_3\\
&\leq \int\limits_{\R^2}\frac{2K_1\,\tau\,d y_2\,dy_3}{\max(1,\|(y_2,y_3)\|_2^{K_2})}\\
&\leq C_{\ref{l: aux 209t}}\,K_1\tau.
\end{align*}
\end{proof}

We assume that the parameters $n,K_1, t,u, M$, $L_u,\offset_u$
are the same as they are in previous subsection, and also set $K_2:=2000$.

\begin{lemma}\label{l: alphaest}
Let $u\in\{0,1,\dots,\lfloor \log_2 n\rfloor\}$, let the $(\alpha,\eta)$--structure be defined as before and fix an $i\leq n$.
Then
\begin{enumerate}

\item For any $(\lambda_1,\lambda_2)\in\Lambda$
and almost every $\omega\in\Event_{i,(\lambda_1,\lambda_2)}$
we have
$$\alpha(i,\omega)\geq \frac{2^{\min(-\lambda_1,L_u)}}{C_{\ref{l: alphaest}}K_1\,t}.$$

\item For $(\lambda_1,\lambda_2)\in\Lambda$
and almost every $\omega\in \Event_{i,(\lambda_1,\lambda_2)}$ we have
$$\alpha(i,\omega)\geq\frac{2^{3\min(-\lambda_1,L_u)+2\min(\lambda_2,L_u)}}{C_{\ref{l: alphaest}}K_1\,\offset_u^2\,t}.$$
\end{enumerate}
Here, $C_{\ref{l: alphaest}}>0$ is a universal constant.
\end{lemma}
\begin{proof}
Fix any $i\leq n$ and a point $\omega=(\omega_1,\omega_2,\dots,\omega_n)\in\Event(t,u)$
such that the linear span of $\{R_j(A(\omega)+M)\}_{j\neq i}$
has dimension $n-1$.
In what follows, given any $\widetilde\omega_i\in\Omega_i$
and $\widetilde\omega:=(\omega_1,\dots,\omega_{i-1},\widetilde\omega_i,\omega_{i+1},\dots,\omega_n)$,
by $A(\widetilde\omega_i)$ (resp, $\rhonew_i$)
we denote the realization of the matrix $A$ (resp., the set $\rhonew_i$)
at $\widetilde\omega$
(observe that $\rhonew_i$ does not depend on $\widetilde\omega_i$, and so our notation is justified).
Now, we take the index $(\lambda_1,\lambda_2)\in\Lambda$ such that $\omega\in \Event_{i,(\lambda_1,\lambda_2)}$
and prove the two assertions of the lemma separately.

\bigskip

First, in view of Lemma~\ref{l: aux 209t}, the probability $\Prob_i$ of the event
$$\big\{\widetilde\omega_i\in\Omega_i:\,\dist(R_i(A(\widetilde\omega_i)+M),H^{i}(A(\widetilde\omega_i)+M))
<2^{\max(\lambda_1,-L_u)+1}\,t\big\}$$
is bounded from above by $C\,K_1\, 2^{\max(\lambda_1,-L_u)}\,t$ for some universal constant $C>0$.
This, together with the definition of $\Event_{i,(\lambda_1,\lambda_2)}$
and the function $\alpha(i,\cdot)$, yields the first assertion of the lemma.

\bigskip

Next, we prove the second assertion of the lemma.
We assume that $\lambda_1\neq +\infty$ and $\lambda_2\neq -\infty$
(otherwise, if $\lambda_1=+\infty$ or $\lambda_2=-\infty$, the assertion is trivial).
In particular, the assumption $\lambda_2\neq-\infty$ implies that $\rhonew_i\neq\emptyset$.
For each unordered pair $e=\{j,k\}\in\rhonew_i$, define an auxiliary event $\widetilde\Event_e\subset\Omega_i$ by
\begin{align*}
\widetilde\Event_{e}
:=\big\{\widetilde\omega_i\in\Omega_i:\,
&\dist(R_i(A(\widetilde\omega_i)+M),H^i(A(\widetilde\omega_i)+M))<2^{\max(\lambda_1,-L_u)+1}\,t
\;\mbox{ and }\;\\
&\dist(R_j(A(\widetilde\omega_i)+M),H^j(A(\widetilde\omega_i)+M))\geq \tau\;\mbox{ and }\;\\
&\dist(R_k(A(\widetilde\omega_i)+M),H^k(A(\widetilde\omega_i)+M))\geq \tau\big\},
\end{align*}
where $\tau:=2^{\min(\lambda_2,L_u)}\,t$.
An application of Lemma~\ref{l: aux9876} gives that
for all $\widetilde\omega_i\in\widetilde\Event_e$ we have
\begin{align}
\max\big(&\dist(R_j(A(\widetilde\omega_i)+M),H^{\{i,j,k\}}(A(\widetilde\omega_i)+M)),\nonumber\\
&\dist(R_k(A(\widetilde\omega_i)+M),H^{\{i,j,k\}}(A(\widetilde\omega_i)+M))\big)\nonumber\\
&\hspace{3cm}\geq \frac{\tau\,\dist(R_i(A(\widetilde\omega_i)+M),H^{\{i,j,k\}}(A(\widetilde\omega_i)+M))}{12\cdot 2^{\max(\lambda_1,-L_u)}\,t}
\label{eq: aux aagof},
\end{align}
while the definition of $\rhonew_i$ implies
\begin{align}
&\hspace{-2cm}\dist(R_i(M),H^{\{i,j,k\}}(A(\widetilde\omega_i)+M))+\offset_u\nonumber
\\\geq\max(&\dist(R_j(A(\widetilde\omega_i)+M),H^{\{i,j,k\}}(A(\widetilde\omega_i)+M)),\nonumber\\
&\dist(R_k(A(\widetilde\omega_i)+M),H^{\{i,j,k\}}(A(\widetilde\omega_i)+M))).\label{eq: aux lkaaer}
\end{align}
Note that $\dist(R_j(A(\widetilde\omega_i)+M),H^{\{i,j,k\}}(A(\widetilde\omega_i)+M))$
and $\dist(R_k(A(\widetilde\omega_i)+M),H^{\{i,j,k\}}(A(\widetilde\omega_i)+M))$ are constant on $\Omega_i$,
so we will just write $\dist(R_j(A+M),H^{\{i,j,k\}}(A+M))$
and $\dist(R_k(A+M),H^{\{i,j,k\}}(A+M))$.

\bigskip

Consider three subcases.
\begin{itemize}

\item If $\max(\dist(R_j(A+M),H^{\{i,j,k\}}(A+M)),\dist(R_k(A+M),H^{\{i,j,k\}}(A+M)))\leq 2\,\offset_u$ on $\Omega_i$
then, by \eqref{eq: aux aagof}, the event $\widetilde\Event_e$ is contained in
\begin{align*}
\big\{\widetilde\omega_i\in\Omega_i:\;&\dist(R_i(A(\widetilde\omega_i)+M),H^i(A(\widetilde\omega_i)+M))< 
2^{\max(\lambda_1,-L_u)+1}\,t
\;\mbox{ and }\;\\
&\dist(R_i(A(\widetilde\omega_i)+M),H^{\{i,j,k\}}(A(\widetilde\omega_i)+M))\leq 24\,\offset_u\cdot 2^{\max(\lambda_1,-L_u)}\,t/\tau\big\},
\end{align*}
and probability $\Prob_i$ of the latter is estimated
by $CK_1 \,\offset_u^2\,2^{3\max(\lambda_1,-L_u)}\,t^3/\tau^2$, in view of Lemma~\ref{l: aux pqhgaf}.

\item If $\max(\dist(R_j(A+M),H^{\{i,j,k\}}(A+M)),\dist(R_k(A+M),H^{\{i,j,k\}}(A+M)))> 2\,\offset_u$ on $\widetilde\Event_e$
and $\tau\geq 4\cdot 12\cdot 2^{\max(\lambda_1,-L_u)}\,t$
then it is not difficult to see from \eqref{eq: aux aagof}--\eqref{eq: aux lkaaer} that $\widetilde\Event_e$ is contained in
\begin{align*}
\big\{\widetilde\omega_i\in\Omega_i:\;&\dist(R_i(A(\widetilde\omega_i)+M),H^i(A(\widetilde\omega_i)+M))
< 2^{\max(\lambda_1,-L_u)+1}\,t
\;\mbox{ and }\;\\
&\dist(R_i(A(\widetilde\omega_i)),H^{\{i,j,k\}}(A(\widetilde\omega_i)+M))\geq \offset_u/2\big\}.
\end{align*}
Estimating the latter event amounts to bounding the probability $\{\xi\in T\setminus B\}$
where $\xi$ is a random vector in $\R^3$ with the distribution density $\rho$ satisfying
$\rho(x)\leq K_1/\max(1,\|x\|_2^{K_2})$, $T$ is a parallel translate of the strip
$\{(x_1,x_2,x_3)\in\R^3:\,0\leq x_1< 4\cdot 2^{\max(\lambda_1,-L_u)}t\}$, and $B$ is the Euclidean ball
$\{x\in\R^3:\,\|x\|_2< \offset_u/2\}$. An easy computation, together with the definition of $\offset_u$, gives:
\begin{align*}
\Prob\{\xi\in T\setminus B\}
&\leq 4\cdot 2^{\max(\lambda_1,-L_u)}t\cdot\sup\limits_{a\geq 0}
\int\limits_{\|(y_1,y_2)\|_2^2\geq \offset_u^2/4-a^2}\frac{K_1\,dy_1\,dy_2}{\max(1,\|(y_1,y_2,a)\|_2^{K_2})}\\
&\leq 8\cdot 2^{\max(\lambda_1,-L_u)}t\cdot
\int\limits_{\R^2}\frac{K_1\,dy_1\,dy_2}{\|(y_1,y_2)\|_2^{K_2}+(\offset_u/4)^{K_2}}\\
&\leq 8\cdot 2^{\max(\lambda_1,-L_u)}t\cdot
\int\limits_{\R^2}\frac{K_1\,dy_1\,dy_2}{\|(y_1,y_2)\|_2^{K_2}+(\offset_u/4)^{K_2}}\\
&\leq CK_1\cdot 2^{\max(\lambda_1,-L_u)}t\cdot (\offset_u/4)^{-K_2+2}\\
&\leq C_0K_1\cdot 2^{\max(\lambda_1,-L_u)}t\cdot 2^{-4L_u}\\
&\leq C_0K_1 \,2^{3\max(\lambda_1,-L_u)}\,t^3/\tau^2,
\end{align*}
where $C,C_0>0$ are some universal constants, and the inequality $c(\offset_u/4)^{-K_2+2}\leq 2^{-4L_u}$
follows from the choice of parameters \eqref{eq: definitions of parameters} and the assumption $K_2= 2000$.
Thus,
$\Prob_i(\widetilde\Event_e)\leq C_0K_1 \,2^{3\max(\lambda_1,-L_u)}\,t^3/\tau^2$.

\item Finally, note that in the situation $\max(\dist(R_j(A+M),H^{\{i,j,k\}}(A+M)),\dist(R_k(A+M),H^{\{i,j,k\}}(A+M)))> 2\,\offset_u$
and $\tau< 4\cdot 12\cdot 2^{\max(\lambda_1,-L_u)}\,t$ we get
$$\Prob_i(\widetilde\Event_e)\leq C'K_1\,2^{\max(\lambda_1,-L_u)}\,t\leq C''K_1\,2^{3\max(\lambda_1,-L_u)}\,t^3/\tau^2$$
just by applying the first part of the proof of the lemma.
\end{itemize}

Summarizing, we have shown that for any $e=(j,k)\in \rhonew_i$
we have
$$\Prob_i\big(\widetilde\Event_e\big)\leq CK_1 \,\offset_u^2\,2^{3\max(\lambda_1,-L_u)}\,t^3/\tau^2$$
for some universal constant $C$.
At the same time, by the definition of the event $\Event_{i,(\lambda_1,\lambda_2)}$, every
point $\widetilde\omega_i\in\Omega_i$ such that $(\omega_1,\dots,\omega_{i-1},\widetilde\omega_i,\omega_{i+1},\dots,\omega_n)
\in\Event_{i,(\lambda_1,\lambda_2)}$
is contained in at least $\lceil |\rhonew_i|/2\rceil$ events $\widetilde \Event_e$ ($e\in\rhonew_i$).
Applying the above upper bounds for $\Prob_i(\widetilde\Event_e)$, Markov's inequality and the definition of $\tau$, we get
\begin{align*}
\Prob_i\big\{&\widetilde\omega_i\in\Omega_i:\,
(\omega_1,\dots,\omega_{i-1},\widetilde\omega_i,\omega_{i+1},\dots,\omega_n)\in\Event_{i,(\lambda_1,\lambda_2)}\big\}\\
&\leq 2\cdot CK_1 \,\offset_u^2\,2^{3\max(\lambda_1,-L_u)}\,t^3/\tau^2\\
&\leq
C'K_1 \,\offset_u^2\,2^{3\max(\lambda_1,-L_u)-2\min(\lambda_2,L_u)}\,t.
\end{align*}
The result follows.
\end{proof}

\begin{lemma}\label{c: sumest}
Let $u\in\{0,1,\dots,\lfloor \log_2 n\rfloor\}$, $L_u$ be as in \eqref{eq: definitions of parameters}, let
the $(\alpha,\eta)$--structure with respect to the event $\Event(t,u)$ be defined as above,
and assume that $(n/2^u)^{1/12-\varepsilon}\geq 16 L_u$. Then
we have
$$\sum_{i=1}^n\frac{\alpha(i,\omega)}{\sharp\eta(i,\omega)}
\geq \frac{2^{u/2-\varepsilon u}n^{-1/2+\varepsilon}}{C_{\ref{c: sumest}}K_1^2\,t}$$
for almost all $\omega\in \Event(t,u)\setminus \bigcup_{p=u+1}^{\lfloor\log_2 n\rfloor}
\Event\big(2^{(p-u)/3}\,t, p\big)$,
where $C_{\ref{c: sumest}}>0$ is a universal constant and
$\Event(\cdot,\cdot)$ is defined in accordance with \eqref{eq: aux pang}.
\end{lemma}
\begin{proof}
Fix a point $\omega\in\Event(t,u)$, and a subset $I\subset[n]$ of cardinality $2^u$ with
$$\dist\big(R_i(A(\omega)+M),H^i(A(\omega)+M)\big)\leq t,\quad\quad i\in I.$$
We will assume that Lemma~\ref{l: alphaest} can be applied to $\omega$ (i.e.\ $\omega$ does not belong to set of measure zero
for which the assertions of the lemma do not hold).
Denote
\begin{align*}
I_0&:=\big\{i\in I:\,\sharp\eta(i,\omega)\leq \max(256\cdot 2^{-L_u/2}n,16)\big\};\\
I_1'&:=\big\{i\in I:\,\sharp\eta(i,\omega)> \max(256\cdot 2^{-L_u/2}n,16)\mbox{ and }
\sharp\eta(i,\omega)< n^{1/2-\varepsilon}2^{u/2+\varepsilon u}\big\};\\
I_1''&:=\big\{i\in I:\,\sharp\eta(i,\omega)> \max(256\cdot 2^{-L_u/2}n,16)\mbox{ and }
\sharp\eta(i,\omega)\geq n^{1/2-\varepsilon}2^{u/2+\varepsilon u}\big\}.
\end{align*}
Observe that, in view of Lemma~\ref{l: rhonew vs eta}, for any $i\in I_1'\cup I_1''$ the set $\rhonew_i(\omega)$
is non-empty.
Let us consider several cases.

\begin{itemize}

\item[(a)] Assume that $|I_0|\geq |I|/2= 2^{u-1}$.
Take any $i_0\in I_0$. Observe that in this case, by the definition of the events $\Event_{i_0,(\lambda_1,\lambda_2)}$,
we have $\omega\in \bigcup\Event_{i_0,(\lambda_1,\lambda_2)}$, where the union is taken over all admissible 
$\lambda_2$ and
$\lambda_1\leq 1$.
Then, by Lemma~\ref{l: alphaest},
we have $\alpha(i_0,\omega)\geq \frac{1}{C\,K_1\,t}$.
Trivially, this gives
$$\sum_{i=1}^n\frac{\alpha(i,\omega)}{\sharp\eta(i,\omega)}\geq 
\sum_{i_0\in I_0}\frac{\alpha(i_0,\omega)}{\max(256\cdot 2^{-L_u/2}n,16)}\geq
\frac{\widetilde c\,2^{u-1}}{\max(256\cdot 2^{-L_u/2}n,16) \,K_1\,t}
\geq \frac{1}{C'\,K_1\,t},$$
where the last inequality follows immediately from the definition of $L_u$.

\item[(b)] Assume that $|I_1'|\geq |I|/4=2^{u-2}$.
This case is treated similarly to (a).
For any $i\in I_1'$ we have, by Lemma~\ref{l: alphaest},
that $\alpha(i,\omega)\geq \frac{1}{C K_1\,t}$.
Hence,
$$\sum_{i=1}^n\frac{\alpha(i,\omega)}{\sharp\eta(i,\omega)}\geq 
\sum_{i\in I_1'}\frac{\alpha(i,\omega)}{n^{1/2-\varepsilon}2^{u/2+\varepsilon u}}\geq
\frac{2^{u-2}\cdot n^{-1/2+\varepsilon}2^{-u/2-\varepsilon u}}{C K_1\,t}.$$

\item[(c)] Assume that $|I_1''|\geq |I|/4=2^{u-2}$. This case is the most complex, and is split into two
subcases.

\begin{itemize}

\item Assume that for at least half of indices $i\in I_1''$ we have
\begin{align*}
\lceil|\rhonew_{i}(\omega)|/2\rceil\kmax(\{\mindist(e),\;e\in\rhonew_{i}(\omega)\})>
\frac{\sqrt{\sharp\eta(i,\omega)}(2^u n)^{1/4}\,t}{n}\Big(\frac{n}{2^u}\Big)^{1/2+\varepsilon}.
\end{align*}
Clearly, for any such $i$ we have $\omega\in \bigcup\Event_{i,(\lambda_1,\lambda_2)}$
with the union taken over all $(\lambda_1,\lambda_2)\in\Lambda$
with $\lambda_1\leq 1$ and $2^{\lambda_2+1}\geq \frac{\sqrt{\sharp\eta(i,\omega)}(2^u n)^{1/4}}{n}(\frac{n}{2^u})^{1/2+\varepsilon}$.
Observe that, in view of the definition of $L_u$,
$$\frac{\sqrt{\sharp\eta(i,\omega)}(2^u n)^{1/4}}{n}\Big(\frac{n}{2^u}\Big)^{1/2+\varepsilon}\leq 2^{L_u}.$$
Then, applying Lemma~\ref{l: alphaest} (this time the second assertion), we get
$$\alpha(i,\omega)\geq \frac{\big(\frac{\sqrt{\sharp\eta(i,\omega)}(2^u n)^{1/4}}{n}(\frac{n}{2^u})^{1/2+\varepsilon}\big)^2
}{C K_1\,\offset_u^2\,t}
=\frac{\sharp\eta(i,\omega)\, n^{-1/2+2\varepsilon}\,2^{-u/2-2\varepsilon u}
}{C K_1\,\offset_u^2\,t}
.$$
Hence, we have for at least $2^{u-3}$ indices $i$:
$$\frac{\alpha(i,\omega)}{\sharp\eta(i,\omega)}
\geq \frac{n^{-1/2+2\varepsilon}\,2^{-u/2-2\varepsilon u}
}{C K_1\,\offset_u^2\,t}.$$
Summing over all such $i$, we obtain
\begin{align*}
\sum_{i=1}^n\frac{\alpha(i,\omega)}{\sharp\eta(i,\omega)}\geq 
\frac{2^{u-3}\cdot n^{-1/2+2\varepsilon}\,2^{-u/2-2\varepsilon u}
}{C K_1\,\offset_u^2\,t}=
\frac{n^{-1/2+2\varepsilon}\,2^{u/2-2\varepsilon u}
}{8C K_1\,\offset_u^2\,t}\geq \frac{n^{-1/2+\varepsilon}\,2^{u/2-\varepsilon u}
}{C' K_1^2\,t},
\end{align*}
where the last inequality follows from the relation $n^{\varepsilon}/2^{\varepsilon u}\geq c\,\offset_u^2/K_1$.

\item Assume that there is $i_1\in I_1''$ with
\begin{align*}
\lceil|\rhonew_{i_1}(\omega)|/2\rceil\kmax(\{\mindist(e),\;e\in\rhonew_{i_1}(\omega)\})\leq
\frac{\sqrt{\sharp\eta(i_1,\omega)}\,(2^u n)^{1/4}\,t}{n}\Big(\frac{n}{2^u}\Big)^{1/2+\varepsilon}.
\end{align*}
Let $J\subset[n]\setminus\{i_1\}$ be the subset of all indices $j$ such that
$$\dist(R_j(A(\omega)+M),H^j(A(\omega)+M))\leq
\frac{\sqrt{\sharp\eta(i_1,\omega)}\,(2^u n)^{1/4}\,t}{n}\Big(\frac{n}{2^u}\Big)^{1/2+\varepsilon}.$$
By the above assumption, at least half of edges (unordered pairs) of $\rhonew_{i_1}(\omega)$
are incident to some indices in $J$.
Then, in view of Lemma~\ref{l: rhonew vs eta} and the condition $\sharp\eta(i_1,\omega)>\max(256\cdot 2^{-L_u/2}n,16)$,
we have $|J|\geq \frac{\sharp\eta(i_1,\omega)}{256L_u^2}$.
On the other hand, the assumption on the magnitude of $n/2^u$ implies $\sharp\eta(i_1,\omega)\geq n^{1/4}\gg 256L_u^2$.
Thus, $\omega$ is contained in the event
$$\Event\Big(\frac{\sqrt{\sharp\eta(i_1,\omega)}\,(2^u n)^{1/4}\,t}{n}\Big(\frac{n}{2^u}\Big)^{1/2+\varepsilon},
\big\lfloor\log_2\big(\sharp\eta(i_1,\omega)/(256L_u^2)\big)\big\rfloor\Big),$$
with $\Event(\cdot,\cdot)$ defined in accordance with \eqref{eq: aux pang}.
Since $\sharp\eta(i_1,\omega)\geq n^{1/2-\varepsilon}2^{u/2+\varepsilon u}$, we obtain
$$\omega\in\bigcup_{p=\lfloor\log_2(n^{1/2-\varepsilon}2^{u/2+\varepsilon u}/(256L_u^2))\rfloor}^{\lfloor\log_2 n\rfloor}
\Event\big(2^{p/2+4}\,2^{-u/4-\varepsilon u}\,n^{-1/4+\varepsilon}L_u\,t,
p\big).$$
By the assumption of the lemma, we have
$\log_2(n^{1/2-\varepsilon}2^{u/2+\varepsilon u}/(256L_u^2))\geq u+1$. Further,
for any $p>u$ the same assumption implies $2^{p/2+4}\,2^{-u/4-\varepsilon u}\,n^{-1/4+\varepsilon}L_u\leq 
2^{p/2}2^{-u/3}n^{-1/6} \leq 2^{(p-u)/3}$.
Thus,
$$\omega\in \bigcup_{p=u+1}^{\lfloor\log_2 n\rfloor}
\Event\big(2^{(p-u)/3}\,t, p\big).$$
The result follows.
\end{itemize}

\end{itemize}
\end{proof}

\bigskip

\begin{proof}[Proof of Theorem~\ref{t:shift}]
Fix parameters $K_1\geq 1$, $K_2:=2000$, a large $n\geq n_0(K_1)$, and let $A$ be an $n\times n$ random matrix
with independent rows, so that each row satisfies condition \eqref{C2} with parameters $K_1,K_2$.
Further, let $M$ be any fixed $n\times n$ matrix. Fix any $\tau\geq 1$.
Recall that the Hilbert--Schmidt norm of an inverse matrix $B^{-1}$ can be written as
$\|B^{-1}\|_{HS}^2=\sum_{i=1}^n \dist(R_i(B),H^i(B))^{-2}$. Further, observe that there is a universal constant
$C>0$ such that for any collection of positive numbers $a_1,a_2,\dots,a_n$ satisfying
$$
|\{i\leq n:\;a_i\leq C\,(2^u/n)^{1/2-\varepsilon/2}\,/\tau\}|<2^u\quad\mbox{for all }u\in\{0,1,\dots,\lfloor\log_2 n\rfloor\},
$$
we necessarily have
$$
\sum_{i=1}^n a_i^{-2}< \tau^2\,n.
$$
Hence,
\begin{align*}
\Prob\big\{&\|(A+M)^{-1}\|_{HS}\geq \tau\,n^{1/2}\big\}\\
&=\Prob\Big\{\sum_{i=1}^n \dist(R_i(A+M),H^i(A+M))^{-2}\geq \tau^2\,n\Big\}\\
&\leq
\Prob\Big\{\exists\, 0\leq u\leq \lfloor\log_2 n\rfloor,\; I\subset [n]:\,|I|=2^u,\mbox{ and }\forall\, i\in I,\\
&\hspace{1cm}\dist(R_i(A+M),H^i(A+M))\leq C\,(2^u/n)^{1/2-\varepsilon/2}\,/\tau\Big\}.
\end{align*}
Now, for any $0\leq u\leq \lfloor \log_2 n\rfloor$ and $s>0$ set
$$\widetilde\Event(s,u):=\Event(s,u)\setminus \bigcup_{p=u+1}^{\lfloor \log_2 n\rfloor}
\Event(2^{p/3-u/3}s,p),$$
where $\Event(\cdot,\cdot)$ are defined by \eqref{eq: aux pang}.
Observe that, in view of Lemmas~\ref{c: sumest} and~\ref{l: alpharho}, we have
$$\Prob(\widetilde \Event(s,u))\leq \frac{C'' K_1^2\,|\Psi|^2\,|\Lambda|\,s}
{2^{u/2-\varepsilon u}n^{-1/2+\varepsilon}}\leq \frac{C' K_1^2\,L_u^4\,s}
{2^{u/2-\varepsilon u}n^{-1/2+\varepsilon}},$$
whenever $(n/2^u)^{1/12-\varepsilon}\geq 16L_u$.
On the other hand, the condition $(n/2^u)^{1/12-\varepsilon}< 16L_u$ necessarily implies that
$\log_2n-u\leq g(K_1)$, where $g(K_1)$ is a function of $K_1$ only. Thus, for all such $u$ we can write
$\Prob(\widetilde \Event(s,u))\leq h(K_1)\,s$, where $h(K_1)$ depends only on $K_1$.
Further,
\begin{align*}
\Prob\Big\{&\exists\, 0\leq u\leq \lfloor \log_2 n\rfloor,\; I\subset [n]:\,|I|=2^u,\mbox{ and }\forall\,i\in I,\\
&\dist(R_i(A+M),H^i(A+M))\leq C\,(2^u/n)^{1/2-\varepsilon/2}\,/\tau\Big\}=\\
\Prob\Big\{&\exists\, 0\leq u\leq \lfloor \log_2 n\rfloor,\; I\subset [n]:\,|I|=2^u,\mbox{ and }\forall\,i\in I,\\
&\dist(R_i(A+M),H^i(A+M))\leq C\,(2^u/n)^{1/2-\varepsilon/2}\,/\tau\mbox{ and }\\
&\big|\big\{j\leq n:\;\dist(R_j(A+M),H^j(A+M))\leq C\,(2^p/n)^{1/2-\varepsilon/2}\,/\tau\big\}\big|<2^p\\
&\mbox{for all }u
+1\leq p\leq \lfloor \log_2 n\rfloor\Big\}\\
&\hspace{5cm}\leq \sum_{u=0}^{\lfloor \log_2 n\rfloor}
\Prob\big(\widetilde\Event\big(C\,(2^u/n)^{1/2-\varepsilon/2}\,/\tau,u\big)\big).
\end{align*}
Combining the last inequality with the estimates for $\Prob(\widetilde \Event(s,u))$, we obtain
\begin{align*}
\Prob\big\{\|(A+M)^{-1}\|_{HS}\geq \tau\,n^{1/2}\big\}
&\leq\sum_{u=0}^{\lfloor \log_2 n\rfloor}
\Prob\big(\widetilde\Event\big(C\,(2^u/n)^{1/2-\varepsilon/2}\,/\tau,u\big)\big)\\
&\leq \sum_{u=0}^{\lfloor \log_2 n\rfloor}
\frac{\widetilde h(K_1)\,L_u^4\,(2^u/n)^{1/2-\varepsilon/2}\,/\tau}
{2^{u/2-\varepsilon u}n^{-1/2+\varepsilon}}\\
&=\sum_{u=0}^{\lfloor\log_2 n\rfloor}\frac{\widetilde h(K_1)\,L_u^4}{\tau}\Big(\frac{2^u}{n}\Big)^{\varepsilon/2}\\
&\leq \frac{v(K_1)}{\tau}.
\end{align*}
The result follows.
\end{proof}

\begin{rem}
The assumptions on the density of $3$--dimensional projections in Theorem~\ref{t:shift} can be relaxed;
in particular, the parameter $K_2$ can be chosen smaller than $2000$ by a more careful computation.
It seems interesting to ask whether the assertion of the main theorem can be proved under only assumption
of bounded density of $3$--dimensional (or $k$--dimensional, for other fixed $k$) projections of the matrix rows.
\end{rem}

\bigskip

{\bf Acknowledgments: } The author would like to thank Nicholas
Cook for referring him to \cite[Lemma~C.1]{TV-limiting},
and anonymous Referees for valuable suggestions.
The research is partially supported by the Simons foundation.
A part of this work was done while the author was the Viterbi postdoctoral fellow
at the Mathematical Sciences Research Institute in Berkeley,
during the Fall 2017 semester.


\begin{thebibliography}{99}

\bibitem{AGLPT}
{
R. Adamczak, O. Guedon, A. E. Litvak, A. Pajor, N. Tomczak-Jaegermann,
Condition number of a square matrix with i.i.d.\ columns drawn from a convex body,
Proc. Amer. Math. Soc. {\bf 140} (2012), no.~3, 987--998. MR2869083
}

\bibitem{ALPT}
{
R. Adamczak, A. E. Litvak, A. Pajor, N. Tomczak-Jaegermann,
Quantitative estimates of the convergence of the empirical covariance matrix in log-concave ensembles,
J. Amer. Math. Soc. {\bf 23} (2010), no.~2, 535--561. MR2601042
}

\bibitem{A et al}
{
M. Aizenman, R. Peled, J. Schenker, M. Shamis, and S. Sodin, Matrix regularizing effects of Gaussian perturbations,
Commun. Contemp. Math. {\bf 19} (2017), no.~3, 1750028, 22 pp. MR3631932
}

\bibitem{BR-smin}
{
A. Basak\ and\ M. Rudelson, Invertibility of sparse non-Hermitian matrices,
Adv. Math. {\bf 310} (2017), 426--483. MR3620692
}

\bibitem{BR-circ}
{
A. Basak, M. Rudelson, The circular law for sparse non-Hermitian matrices. arXiv:1707.03675
}

\bibitem{BP}
{
G. Ben Arous\ and\ S. P\'ech\'e, Universality of local eigenvalue statistics for some sample covariance matrices,
Comm. Pure Appl. Math. {\bf 58} (2005), no.~10, 1316--1357. MR2162782
}

\bibitem{BC12}
{
C. Bordenave\ and\ D. Chafa\"i, Around the circular law, Probab. Surv. {\bf 9} (2012), 1--89. MR2908617
}

\bibitem{Bourgain}
{
J. Bourgain, On a problem of Farrell and Vershynin in random matrix theory, in {\it Geometric aspects of functional analysis}, 65--69, Lecture Notes in Math., 2169, Springer, Cham. MR3645115
}

\bibitem{BGVV}
{
S. Brazitikos, A. Giannopoulos, P. Valettas, B.-H. Vritsiou,
{\it Geometry of isotropic convex bodies}, Mathematical Surveys and Monographs, 196,
American Mathematical Society, Providence, RI, 2014. MR3185453
}

\bibitem{C-smin}
{
N. Cook, Lower bounds for the smallest singular value of structured random matrices. arXiv:1608.07347
}

\bibitem{Edelman}
{
A. Edelman, Eigenvalues and condition numbers of random matrices,
SIAM J. Matrix Anal. Appl. {\bf 9} (1988), no.~4, 543--560. MR0964668
}

\bibitem{FV}
{
B. Farrell\ and\ R. Vershynin, Smoothed analysis of symmetric random matrices with continuous distributions, Proc. Amer. Math. Soc. {\bf 144} (2016), no.~5, 2257--2261. MR3460183
}

\bibitem{KKS}
{
J. Kahn, J. Koml\'os\ and\ E. Szemer\'edi, On the probability that a random $\pm 1$-matrix is singular,
J. Amer. Math. Soc. {\bf 8} (1995), no.~1, 223--240. MR1260107
}

\bibitem{K}
{
J. Koml\'os, On the determinant of $(0,\,1)$ matrices, Studia Sci. Math. Hungar {\bf 2} (1967), 7--21. MR0221962
}

\bibitem{L}
{
R. Lata\l a, Some estimates of norms of random matrices, Proc. Amer. Math. Soc.
{\bf 133} (2005), no.~5, 1273--1282 (electronic). MR2111932
}

\bibitem{LPRT}
{
A. E. Litvak, A. Pajor, M. Rudelson, N. Tomczak-Jaegermann,
Smallest singular value of random matrices and geometry of random polytopes,
Adv. Math. {\bf 195} (2005), no.~2, 491--523. MR2146352
}

\bibitem{Neumann}
{
J. von Neumann\ and\ H. H. Goldstine, Numerical inverting of matrices of high order,
Bull. Amer. Math. Soc. {\bf 53} (1947), 1021--1099. MR0024235
}

\bibitem{Paouris}
{
G. Paouris, Concentration of mass on convex bodies, Geom. Funct. Anal. {\bf 16} (2006), no.~5, 1021--1049. MR2276533
}

\bibitem{RT}
{
E. Rebrova\ and\ K. Tikhomirov, Coverings of random ellipsoids, and invertibility of matrices with i.i.d. heavy-tailed entries, Israel J. Math.,
to appear. arXiv:1508.06690
}

\bibitem{R}
{
M. Rudelson, Invertibility of random matrices: norm of the inverse, Ann. of Math. (2) {\bf 168} (2008), no.~2, 575--600. MR2434885
}

\bibitem{RV_square}
{
M. Rudelson\ and\ R. Vershynin, The Littlewood--Offord problem and invertibility of random matrices,
Adv. Math. {\bf 218} (2008), no.~2, 600--633. MR2407948
}

\bibitem{RV_JAMS}
{
M. Rudelson\ and\ R. Vershynin, Invertibility of random matrices: unitary and orthogonal perturbations,
J. Amer. Math. Soc. {\bf 27} (2014), no.~2, 293--338. MR3164983
}

\bibitem{RV_images}
{
M. Rudelson\ and\ R. Vershynin, Small ball probabilities for linear images of high-dimensional distributions,
Int. Math. Res. Not. IMRN {\bf 2015}, no.~19, 9594--9617. MR3431603
}

\bibitem{SST}
{
A. Sankar, D. A. Spielman\ and\ S.-H. Teng, Smoothed analysis of the condition numbers and growth factors of matrices,
SIAM J. Matrix Anal. Appl. {\bf 28} (2006), no.~2, 446--476 (electronic). MR2255338
}

\bibitem{Smale}
{
S. Smale, On the efficiency of algorithms of analysis, Bull. Amer. Math. Soc. (N.S.) {\bf 13} (1985),
no.~2, 87--121. MR0799791
}

\bibitem{ST}
{
D. A. Spielman\ and\ S.-H. Teng, Smoothed analysis of algorithms, in {\it Proceedings of the International Congress of Mathematicians,
Vol. I (Beijing, 2002)}, 597--606, Higher Ed. Press, Beijing. MR1989210
}

\bibitem{TV}
{
T. Tao\ and\ V. H. Vu, Inverse Littlewood-Offord theorems and the condition number of random discrete matrices,
Ann. of Math. (2) {\bf 169} (2009), no.~2, 595--632. MR2480613
}

\bibitem{TV-limiting}
{
T. Tao\ and\ V. Vu, Random matrices: the distribution of the smallest singular values,
Geom. Funct. Anal. {\bf 20} (2010), no.~1, 260--297. MR2647142
}

\bibitem{TV-cn}
{
T. Tao\ and\ V. Vu, Smooth analysis of the condition number and the least singular value,
Math. Comp. {\bf 79} (2010), no.~272, 2333--2352. MR2684367
}

\bibitem{YBK}
{
Y. Q. Yin, Z. D. Bai\ and\ P. R. Krishnaiah, On the limit of the largest eigenvalue
of the large-dimensional sample covariance matrix,
Probab. Theory Related Fields {\bf 78} (1988), no.~4, 509--521. MR0950344
}

\end{thebibliography}
\end{document}